\documentclass[12pt]{article}
\usepackage{fullpage, amsthm, amsmath, amssymb, url}
\usepackage[all]{xy}

\newtheorem{theorem}{Theorem}[section]
\newtheorem{lemma}[theorem]{Lemma}
\newtheorem{cor}[theorem]{Corollary}

\theoremstyle{definition}
\newtheorem{defn}[theorem]{Definition}
\newtheorem{hypothesis}[theorem]{Hypothesis}
\newtheorem{example}[theorem]{Example}
\newtheorem{remark}[theorem]{Remark}

\numberwithin{equation}{theorem}

\def\CC{\mathbb{C}}
\def\Fp{\mathbb{F}_p}
\def\Qp{\mathbb{Q}_p}
\def\QQ{\mathbb{Q}}
\def\RR{\mathbb{R}}
\def\Zp{\mathbb{Z}_p}
\def\ZZ{\mathbb{Z}}
\def\calH{\mathcal{H}}
\def\calM{\mathcal{M}}
\def\gothm{\mathfrak{m}}
\def\gotho{\mathfrak{o}}

\DeclareMathOperator{\Aut}{Aut}
\DeclareMathOperator{\charac}{char}
\DeclareMathOperator{\Frac}{Frac}

\DeclareMathOperator{\id}{id}

\DeclareMathOperator{\perf}{perf}
\DeclareMathOperator{\spect}{sp}

\title{Nonarchimedean geometry of Witt vectors}
\author{Kiran S. Kedlaya}
\date{February 15, 2012}

\begin{document}

\maketitle

\begin{abstract}
Let $R$ be a perfect $\Fp$-algebra, equipped with the trivial norm.
Let $W(R)$ be the ring of $p$-typical Witt vectors over $R$, equipped with the
$p$-adic norm. At the level of nonarchimedean analytic spaces (in the sense
of Berkovich), we demonstrate a close analogy between $W(R)$ and the
polynomial ring $R[T]$ equipped with the Gauss norm, in which the role of the
structure morphism from $R$ to $R[T]$ is played by the Teichm\"uller map.
For instance, 
we show that
the analytic space associated to $R$ is a strong deformation
retract of the space associated to $W(R)$.
We also show that each fibre forms a tree under the relation of 
pointwise comparison, and classify the points of fibres 
in the manner of Berkovich's classification of points of a nonarchimedean
disc. Some results pertain to
the study of $p$-adic representations of \'etale fundamental
groups of nonarchimedean analytic spaces (i.e., relative $p$-adic Hodge theory).
\end{abstract}

\section*{Introduction}

There now exist several approaches to nonarchimedean analytic geometry, including rigid analytic geometry
(Tate), formal geometry (Raynaud), and adic geometry (Huber, Fujiwara-Kato). However, the approach
exhibiting the closest links with classical topology is that of Berkovich \cite{berkovich1}.
Berkovich overcomes the lack of connectivity of nonarchimedean topologies 
by considering spaces of
multiplicative seminorms, using an analogue of the usual Gel'fand transform
relating commutative Banach algebras to spaces of continuous functions
on compact topological spaces.

Much is known about the topology of Berkovich analytic spaces.
For instance, Berkovich showed that smooth analytic spaces over a field are 
locally contractible
\cite{berkovich-contract1, berkovich-contract2}.
More recently, Payne \cite{payne} showed that the analytification of an algebraic
variety over a field can be viewed as an inverse limit of finite polyhedral complexes;
separately, Hrushovski and Loeser \cite{hrushovski-loeser} have used model-theoretic techniques
to show that such analytifications
are locally contractible and retract onto finite CW-complexes. One can also relate homotopy types of analytic spaces to
degenerations; for instance, the analytification
of a semistable curve over a complete discretely valued field has the same homotopy type as the
graph of the special fibre of a minimal proper regular model over the valuation subring. This relationship,
and its link to mixed Hodge structures, has been further pursued by Berkovich
\cite{berkovich-hodge} and Nicaise \cite{nicaise}.

In this paper, we consider degenerations in mixed characteristic from the point of
view of Witt vectors.  Recall that for $p$ a fixed prime number, the $p$-typical Witt vector functor converts perfect
$\Fp$-algebras into $p$-adically complete algebras. 
Let $R$ be a perfect $\Fp$-algebra equipped with the trivial norm, and 
equip the associated Witt vector ring
$W(R)$ with the $p$-adic norm.
Let $\calM(R)$ and $\calM(W(R))$ denote the resulting Berkovich spaces.
There is a natural multiplicative map $R \to W(R)$ given by Teichm\"uller lifting;
this map is not a ring homomorphism, but it nonetheless induces a restriction map
$\mu: \calM(W(R)) \to \calM(R)$ as if it were a homomorphism. 

We establish several results that liken the relationship between $\calM(R)$
and $\calM(W(R))$ to the relationship between $\calM(R)$ and $\calM(R[T])$
when $R[T]$ carries the Gauss norm (i.e.,
between a base space and a disc bundle over the base). We first 
construct a continuous section
$\lambda: \calM(R) \to \calM(W(R))$ giving 
a maximal lifting of a seminorm on $R$ to $W(R)$;
this identifies $\calM(R)$ as a retract of $\calM(W(R))$.
We then refine this calculation
to show (Theorem~\ref{T:final homotopy}) that $\calM(R)$
is a strong deformation retract of $\calM(W(R))$, 
and that any subset of 
$\calM(R)$
has the same homotopy type as its inverse image under the projection $\mu$.
We finally describe the geometry of the fibres
of the projection map $\mu$: each fibre may be naturally viewed as a tree
in both a topological fashion (as an inverse limit of finite contractible
one-dimensional simplicial complexes)
and a combinatorial fashion 
(as a partially ordered set in which any two elements dominating
a common third element are comparable). 

The motivation for this work is to describe $p$-adic Hodge
theory (the study of continuous representations of Galois groups
of finite extensions of the $p$-adic field $\Qp$) in a fashion
that permits consideration also of \'etale fundamental groups
of analytic spaces. A preview of such a description is given in
\cite{kedlaya-icm}, together with an application to the
construction of local systems on Rapoport-Zink period spaces;
that preview already includes a few of the results given here, including
the definition of the maps $\lambda$ and $\mu$.
A related development is a
reformulation of $p$-adic Hodge theory by Fargues and Fontaine
\cite{fargues-fontaine}, in which one works with
coherent sheaves on an object constructed from a ring of Witt vectors,
which behaves formally like an analytic curve.

\subsection*{Acknowledgments}
Thanks to Laurent Fargues and Jean-Marc Fontaine for providing
a preliminary draft of \cite{fargues-fontaine}, to
Michael Temkin for helpful discussions, and to Peter Scholze and Liang Xiao
for feedback on early drafts of this paper.
The author was supported by
NSF (CAREER grant DMS-0545904), DARPA (grant HR0011-09-1-0048),
MIT (NEC Fund, Cecil and Ida Green professorship), 
and IAS (NSF grant DMS-0635607, James D. Wolfensohn Fund).

\section{Nonarchimedean analytic spaces}

We begin by setting notation and terminology concerning
nonarchimedean analytic spaces, as in \cite{berkovich1}.

\begin{defn}
Consider the following conditions on an abelian group $G$ and a
function $\alpha: G \to [0, +\infty)$.
\begin{enumerate}
\item[(a)]
For all $g,h \in G$, we have $\alpha(g-h) \leq \max\{\alpha(g), \alpha(h)\}$.
\item[(b)]
We have $\alpha(0) = 0$.
\item[(b$'$)]
For all $g \in G$, we have $\alpha(g) = 0$ if and only if $g=0$.
\end{enumerate}
We say $\alpha$ is a \emph{seminorm} if it satisfies 
(a) and (b), and a
\emph{norm} if it satisfies (a) and (b$'$). 
These would typically be called \emph{nonarchimedean} seminorms and norms,
but we will use no other kind in this paper.

If $\alpha, \alpha'$ 
are two seminorms on the same abelian group $G$, 
we say $\alpha$ \emph{dominates} $\alpha'$, and write 
$\alpha \geq \alpha'$ or $\alpha' \leq \alpha$, if there exists
$c \in (0, +\infty)$ for which $\alpha'(g) \leq c \alpha(g)$ for all $g \in G$.
If $\alpha$ and $\alpha'$ dominate each other, we say they are 
\emph{equivalent}; in this case, $\alpha$ is a norm if and only if
$\alpha'$ is.

Let $G,H$ be two abelian groups equipped with seminorms
$\alpha, \beta$.
We say a homomorphism $\phi: G \to H$ is \emph{bounded} if 
$\alpha$ dominates $\beta \circ \phi$.
We say $\phi$ is \emph{isometric} if $\alpha = \beta \circ \phi$.
\end{defn}

\begin{defn}
Let $\alpha$ be a seminorm on an abelian group $G$.
For any subgroup $H$ of $G$, $\alpha$ induces a \emph{quotient
seminorm} on $G/H$ defined by
\[
g + H \mapsto \inf\{\alpha(g+h): h \in H\}. 
\]
This defines a norm if $H$ is closed; for
instance, the quotient seminorm on $G/\ker(\alpha)$ is a norm.

The group of Cauchy sequences in $G$ carries a seminorm defined by
\[
(x_0,x_1,\dots) \mapsto \lim_{i \to \infty} \alpha(x_i).
\]
Passing to
the quotient by the kernel of this seminorm gives the \emph{separated completion} 
$\widehat{G}$ of $G$. The map $G\to \widehat{G}$ given by
$x \mapsto x,x,\dots$ is an isometry, and hence injective if $\alpha$
itself is a norm; in that case, we call $\widehat{G}$ simply the
\emph{completion} of $G$.
\end{defn}

\begin{defn} \label{D:submultiplicative}
Let $A$ be a ring. Consider the following conditions
on a multiplicative seminorm $\alpha$ on the additive group of $A$. 
\begin{enumerate}
\item[(c)]
We have $\alpha(1) \leq 1$, and for all $g,h \in A$, we have $\alpha(gh) \leq \alpha(g) \alpha(h)$.
\item[(c$'$)]
We have (c), and for all $g \in A$ we have $\alpha(g^2) = \alpha(g)^2$.
\item[(c$''$)]
We have (c), and for all $g,h \in A$, we have $\alpha(gh) = \alpha(g) \alpha(h)$.
\end{enumerate}
We say $\alpha$ is 
\emph{submultiplicative} if it satisfies (c),
\emph{power-multiplicative} if it satisfies (c$'$),
and \emph{multiplicative} if it satisfies (c$''$). 
We make some quick observations about these definitions.
\begin{itemize}
\item[(i)] 
For $\alpha$ a submultiplicative seminorm, $\alpha(1) = 1$ unless $\alpha$ is identically zero.
\item[(ii)]
Any power-multiplicative seminorm 
$\alpha$ satisfies $\alpha(g^n) = \alpha(g)^n$ for 
all $g \in A$ and all nonnegative integers $n$.
\item[(iii)]
Any multiplicative seminorm is power-multiplicative.
\item[(iv)]
If $\alpha$ is a submultiplicative seminorm and
$\alpha'$ is a power-multiplicative seminorm,
then $\alpha \geq \alpha'$ 
if and only if $\alpha(a) \geq \alpha'(a)$ for all
$a \in A$.
\end{itemize}
\end{defn}

\begin{example}
For any abelian group $G$, the \emph{trivial norm} on $G$ sends 
$0$ to $0$ and any nonzero $g \in G$ to 1. For any nonzero ring $A$, the trivial
norm on $A$ is submultiplicative in all cases,
power-multiplicative if and only if $A$ is reduced,
and multiplicative if and only if $A$ is  an integral domain.
(The trivial norm on the zero ring is multiplicative by virtue of the fact that we do not
force $\alpha(1) = 1$.)
\end{example}

\begin{defn} \label{D:residue ring}
For $A$ a ring equipped with a submultiplicative seminorm $|\cdot|$,
we write
\begin{align*}
\gotho_A &= \{x \in A: |x| \leq 1\} \\
\gothm_A &= \{x \in A: |x| < 1\} \\
\kappa_A &= \gotho_A/\gothm_A.
\end{align*}
If $A$ is a field and $|\cdot|$ is a multiplicative norm, then 
$\kappa_A$ is also a field. (The field $\kappa_A$ is normally called the
\emph{residue field} of $A$, but we will use this term mostly for something
else; see Definition~\ref{D:spectrum}.)
\end{defn}

\begin{defn}
Let $A$ be a ring equipped with a submultiplicative (semi)norm $|\cdot|$.
For $r \geq 0$, the \emph{$r$-Gauss (semi)norm} $|\cdot|_r$ on $A[T]$
(for the generator $T$, when this needs to be specified)
is the submultiplicative (semi)norm defined by the formula
\begin{equation} \label{eq:Gauss formula}
 \left| \sum_i x_i T^i \right|_r = \max_i \{|x_i| r^i\};
\end{equation}
this is multiplicative if $|\cdot|$ is multiplicative
(see Lemma~\ref{L:Gauss multiplicative}).
We refer to the 1-Gauss (semi)norm also simply as the \emph{Gauss (semi)norm}
(or \emph{Gauss extension}).
\end{defn}

\begin{lemma} \label{L:Gauss multiplicative}
Let $A$ be a ring equipped with a multiplicative seminorm $|\cdot|$.
Then for any $r \geq 0$, the $r$-Gauss seminorm on $A[T]$ is multiplicative.
\end{lemma}
\begin{proof}
This is evident in case $r=0$, so assume $r > 0$.
Let $a = \sum_j a_j T^j, b = \sum_k b_k T^k$ be any two elements of $A[T]$.
Choose the smallest indices $j,k$ for which $|a_j| r^j, |b_k| r^k$ are maximized,
and put $i = j+k$. The coefficient of $T^i$ in $ab$ is then equal to $a_j b_k$ plus the sum of
$a_{j'} b_{k'}$ over all pairs $(j', k') \neq (j,k)$ for which $j'+k' = j+k$. For each such
pair, either $j' < j$, in which case
\[
|a_{j'}| r^{j'} < |a_j| r^j, \qquad |b_{k'}| r^{k'} \leq |b_k| r^k, 
\]
or $k' < k$, in which case
\[
|a_{j'}| r^{j'} \leq |a_j| r^j, \qquad |b_{k'}| r^{k'} < |b_k| r^k. 
\]
In both cases, we conclude that $|a_{j'} b_{k'}| < |a_j b_k|$, so the coefficient of $T^i$
in $ab$ has norm $|a_j b_k|$. This forces $|ab|_r = |a|_r |b|_r$, as desired.
\end{proof}

\begin{remark} \label{R:construct Gauss by quotient}
For each $z \in A$, one has a $r$-Gauss seminorm on $A[T]$ for the generator
$T-z$. This seminorm can also be constructed by equipping $A[T]$
with the $s$-Gauss norm for some $s \geq \max\{r,|z|\}$, forming the $r$-Gauss extension
to $A[T][U]$, then passing to the quotient norm on $A[T][U]/(U - T + z) 
\cong A[T]$.
\end{remark}

\begin{defn} \label{D:spectrum}
Let $A$ be a ring equipped with a submultiplicative norm $|\cdot|$. The \emph{Gel'fand spectrum}
$\calM(A)$ of $A$ is the set of multiplicative seminorms $\alpha$ on $A$ 
dominated by $|\cdot|$,
topologized as a closed subspace of the product $\prod_{a \in A} [0, |a|]$
(see observation (iv) of Definition~\ref{D:submultiplicative}).
In particular, $\calM(A)$ is compact by Tikhonov's theorem.
A subbasis of the topology on $\calM(A)$
is given by the sets $\{\alpha \in \calM(A): \alpha(f) \in I\}$
for each $f \in A$ and each open interval $I \subseteq \RR$.
For any bounded homomorphism $\phi: A \to B$ between rings equipped with 
submultiplicative norms, restriction along $\phi$ defines a 
continuous map $\phi^*: \calM(B)
\to \calM(A)$; this map is a homeomorphism when $B = \widehat{A}$.

For $\alpha \in \calM(A)$, the seminorm $\alpha$ induces a multiplicative norm on the integral domain $A/\alpha^{-1}(0)$,
and hence also on $\Frac(A/\alpha^{-1}(0))$. The completion of this latter field is the \emph{residue field}
of $\alpha$, denoted $\calH(\alpha)$. (Note that $\calH(\alpha)$ itself has
a ``residue field'' $\kappa_{\calH(\alpha)}$ in the sense of
Definition~\ref{D:residue ring}.)
\end{defn}

\begin{theorem} \label{T:Gelfand transform1}
Let $A$ be a nonzero ring equipped with a submultiplicative norm $|\cdot|$. 
Then $\calM(A) \neq \emptyset$.
\end{theorem}
\begin{proof}
Replace $A$ by its completion, then apply \cite[Theorem~1.2.1]{berkovich1}.
\end{proof}

\begin{defn} \label{D:transform}
Let $A$ be a ring equipped with a submultiplicative norm $|\cdot|$.
Let $|\cdot|_{\sup}: \prod_{\alpha \in \calM(A)} \calH(\alpha) \to [0, +\infty]$
denote the supremum of the norms on the $\calH(\alpha)$.
Let $P$ be the inverse image of $[0, +\infty)$ under $|\cdot|_{\sup}$;
then $|\cdot|_{\sup}$ defines a power-multiplicative norm on $P$.
The diagonal map $A \to \prod_{\alpha \in \calM(A)} \calH(\alpha)$ then 
factors through a bounded homomorphism $A \to P$, called the
\emph{Gel'fand transform} of $A$.
\end{defn}

\begin{lemma} \label{L:spectral seminorm}
Let $A$ be a ring equipped with a submultiplicative norm $|\cdot|$.
Then the restriction of $|\cdot|_{\sup}$ to $A$ along the Gel'fand
transform computes the \emph{spectral seminorm}
$|a|_{\spect} = \lim_{s \to \infty} |a^s|^{1/s}$ on $A$.
\end{lemma}
\begin{proof}
See \cite[Theorem~1.3.1]{berkovich1}.
\end{proof}

\begin{remark}
Let $A$ be a ring equipped with a submultiplicative norm $|\cdot|$.
Let $I \subset A$ be the kernel of the spectral seminorm.
Choose a closed subset $X$ of $\calM(A)$.
Let $S$ be the multiplicative subset of $a \in A/I$ for which
$\inf\{\alpha(a): \alpha \in X\} > 0$. Put $B = S^{-1} (A/I)$, equipped
with the supremum norm over $X$; then the map $A \to B$ is a 
bounded homomorphism inducing a homeomorphism of $\calM(B)$
with a closed subset of $\calM(A)$ containing $X$. 
In many (but not all) cases, this closed subset equals $X$;
for instance, this occurs for the sets described in Definition~\ref{D:rational}
below. This is related to the nonarchimedean analogues of the
notions of \emph{holomorphically convex} and \emph{meromorphically convex}
sets; see \cite[\S 2.6]{berkovich1}.
\end{remark}

The following examples of the previous construction occur 
when comparing nonarchimedean
analytic geometry to formal geometry or rigid analytic geometry
(as explained in \cite{berkovich2}).
\begin{defn} \label{D:rational}
Let $A$ be a ring equipped with a submultiplicative norm $|\cdot|$.
A \emph{Weierstrass subspace} of $\calM(A)$ is a closed subspace of the form
\[
U = \{\alpha \in \calM(A): \alpha(f_i) \leq p_i \quad
(i=1,\dots,n)\}
\]
for some $f_1,\dots,f_n \in A$ and some
$p_1,\dots,p_n > 0$. 
A \emph{Laurent subspace} of $\calM(A)$ is a closed subspace of the form
\[
U = \{\alpha \in \calM(A): \alpha(f_i) \leq p_i, \quad \alpha(g_j) \geq q_j \quad
(i=1,\dots,m; \, j=1,\dots,n)\}
\]
for some $f_1,\dots,f_m,g_1,\dots,g_n \in A$ and some
$p_1,\dots,p_m,q_1,\dots,q_n > 0$; the Laurent subspaces form a basis
of closed neighborhoods for the topology of $\calM(A)$.
A \emph{rational subspace} of $\calM(A)$ is a closed subspace of the form
\[
U = \{\alpha \in \calM(A): \alpha(f_i) \leq p_i \alpha(g) \quad
(i=1,\dots,n)\}
\]
for some $f_1,\dots,f_n,g \in A$ which generate the unit ideal 
in $\widehat{A}$ and some $p_1,\dots,p_n > 0$;
we may assume without loss of generality that $f_n = g, p_n = 1$.
The intersection of rational subspaces is rational 
\cite[Proposition~7.2.3/7]{bgr}; consequently, any Laurent
subspace is rational.

We will say a continuous
map between Gel'fand spectra
is \emph{strongly continuous} if the inverse image of any 
Weierstrass (resp.\ Laurent, rational) subspace
is a finite union of Weierstrass (resp.\ Laurent, rational) subspaces.
For instance, the restriction map along a bounded homomorphism 
is strongly continuous.
\end{defn}

\begin{remark} \label{R:rational}
It is easy to see that a Weierstrass or
Laurent subspace of $\calM(\widehat{A})$ remains Weierstrass or Laurent
when viewed as a subset of $\calM(A)$. This is also true for rational subspaces,
but the argument is a bit less immediate.
Let
\[
U = \{\alpha \in \calM(\widehat{A}): \alpha(f_i) \leq p_i \alpha(g) \quad
(i=1,\dots,n)\}
\]
be a rational subspace of $\calM(\widehat{A})$
for some $f_1,\dots,f_n,g \in \widehat{A}$ 
which generate the unit ideal and some
$p_1,\dots,p_n > 0$. Choose $u_1,\dots,u_n,v \in \widehat{A}$ for which
$u_1 f_1 + \cdots + u_n f_n + v g = 1$. For $\alpha \in U$, we then have
\[
1 \leq \max\{\alpha(u_1f_1), \dots, \alpha(u_nf_n), \alpha(vg) \} \\
\leq \max\{|u_1| p_1, \dots, |u_n| p_n, |v|\} \alpha(g).
\]
Choose $\epsilon \in (0,1)$ so that $\epsilon \max\{|u_1| p_1, \dots, 
|u_n| p_n, |v|\} < 1$; then $\alpha(g) \geq \epsilon$ for all $\alpha \in U$.
Choose $f'_1,\dots,f'_n,g' \in A$ with
\[
|f_1 - f'_1| < p_1 \epsilon, \dots,
|f_n - f'_n| < p_n \epsilon, |g-g'| < \epsilon.
\]
On one hand,
\[
|u_1 f'_1 + \cdots + u_n f'_n + v g' - 1| \leq
\epsilon \max\{|u_1| p_1, \dots, |u_n| p_n, |v|\} < 1,
\]
so $f'_1,\dots,f'_n,g'$ still generate the unit ideal in $\widehat{A}$.
On the other hand,
\[
U = \{\alpha \in \calM(\widehat{A}): \alpha(f'_i) \leq p_i \alpha(g') \quad
(i=1,\dots,n)\},
\]
so $U$ is a rational subspace of $\calM(A)$.
(Note that we cannot hope to ensure that $f'_1,\dots,f'_n,g'$ generate
the unit ideal in $A$ itself.)
\end{remark}

\begin{defn} \label{D:tensor}
Let $A,B,C$ be rings equipped with submultiplicative norms $|\cdot|_A,
|\cdot|_B, |\cdot|_C$. Let $A \to B$ and $A \to C$ be bounded homomorphisms.
Define the \emph{product seminorm} $|\cdot|_{B \otimes C}$ on $B \otimes_A C$
by taking $|f|_{B \otimes C}$ to be the infimum of $\max_i \{|b_i|_B |c_i|_C\}$
over all presentations $\sum_i b_i \otimes c_i$ of $f$.
Let $B \widehat{\otimes}_A C$ be the separated completion of $B \otimes_A C$
for the product seminorm.

It is sometimes difficult to tell whether or not $B \widehat{\otimes}_A C$ is nonzero;
we get around this using the following definition.
By a \emph{splitting} of $\iota: A \to B$, we will mean
a bounded homomorphism $\pi: B \to A$ of $A$-modules with $\pi \circ \iota = \id_A$.
We say $\iota$ is \emph{split} if it admits a splitting; this implies that $|\cdot|_A$ is equivalent
to the restriction of $|\cdot|_B$.
\end{defn}

\begin{lemma} \label{L:split}
Let $A,B,C$ be rings equipped with submultiplicative norms $|\cdot|_A,
|\cdot|_B, |\cdot|_C$. Let $A \to B$ and $A \to C$ be bounded homomorphisms.
Then $|\cdot|_{B \otimes C}$ induces a 
submultiplicative norm
on $B \widehat{\otimes}_A C$.
In addition, if $\iota: A \to B$ is split, then so is $C \to B \widehat{\otimes}_A C$.
\end{lemma}
\begin{proof}
From the presentation $1 = 1 \otimes 1$, we read off that
$|1|_{B \otimes C} \leq 1$.
For $f = \sum_i b_i \otimes c_i, f' = \sum_j b'_j \otimes c'_j \in B \otimes_A C$,
we may write
$f f' = \sum_{i,j} (b_i b'_j) \otimes (c_i c'_j)$, and deduce that
\[
|ff'|_{B \otimes C} \leq \max_{i,j} \{ |b_i b'_j|_B |c_i c'_j|_C\}
\leq \max_i \{|b_i|_B |c_i|_C\} \max_j \{|b'_j|_B |c'_j|_C\}.
\]
Taking the infimum over all presentations of $f$ and $f'$ yields
$|ff'|_{B \otimes C} \leq |f|_{B \otimes C} |f'|_{B \otimes C}$,
so  $|\cdot|_{B \otimes C}$ is a submultiplicative norm on 
$B \widehat{\otimes}_A C$. 

Suppose further that $\pi: B \to A$ is a splitting of $\iota$.
By tensoring $\pi$ over $A$ with $C$, we obtain a 
bounded projection $B \otimes_A C \to C$
of $C$-modules with $C \to B \otimes_A C \to C$ being the identity.
By continuity, we obtain a projection $B \widehat{\otimes}_A C \to C$
with the same effect, so $C \to B \widehat{\otimes}_A C$ is split.
\end{proof}

\begin{remark}
Keep in mind that $|\cdot|_{B \otimes C}$ need not be multiplicative 
even if $|\cdot|_A, |\cdot|_B, |\cdot|_C$ are
multiplicative. For example, if $K$ is a quadratic extension of the $p$-adic field
$\QQ_p$, then $K \otimes_{\QQ_p} K \cong K \widehat{\otimes}_{\QQ_p} K$ splits as a direct sum
of two copies of $K$.
\end{remark}

\begin{lemma} \label{L:surjective}
Let $A,B$ be rings equipped with submultiplicative norms $|\cdot|_A$, $|\cdot|_B$.
Let $\phi: A \to B$ be a split homomorphism.
Then the map $\phi^*$ is surjective.
\end{lemma}
\begin{proof} 
For any $\alpha \in \calM(A)$, the homomorphism $\calH(\alpha) \to B \widehat{\otimes}_A \calH(\alpha)$ is split by Lemma~\ref{L:split}; in particular, the target is nonzero and carries a submultiplicative
norm. By Theorem~\ref{T:Gelfand transform1},
there exists some $\beta \in \calM(B \widehat{\otimes}_A \calH(\alpha))$.
The seminorms $\alpha$ and $\beta \circ \phi$ on $A$ then coincide.
\end{proof}

\begin{lemma} \label{L:surjective2}
Let $A,B,C$ be rings equipped with submultiplicative norms. Then the map $\calM(B \widehat{\otimes}_A C) 
\to \calM(B) \times_{\calM(A)} \calM(C)$ is surjective.
\end{lemma}
\begin{proof}
Choose $\beta \in \calM(B)$, $\gamma \in \calM(C)$ having the same 
image $\alpha$ in
$\calM(A)$.
Using the existence of Schauder bases for Banach modules over nonarchimedean fields,
it can be shown that the completed tensor product of any two nonzero Banach modules
over $\calH(\alpha)$ is nonzero \cite[Lemma~1.3.11]{kedlaya-course}.
In particular, $D = \calH(\beta) \widehat{\otimes}_{\calH(\alpha)} \calH(\gamma)$
is nonzero, so by Theorem~\ref{T:Gelfand transform1}, there exists some $\delta \in \calM(D)$.
The restrictions of $\delta$ to $\calH(\beta), \calH(\gamma)$ give back
$\beta, \gamma$, so the same is true of the restrictions to $B,C$.
\end{proof}

\section{Nonarchimedean geometry of polynomial rings}
\label{sec:polynomial}

To illustrate the results we have concerning the nonarchimedean geometry of Witt vectors, we
first describe the analogous statements relating the nonarchimedean 
analytic spaces associated to a ring $R$ and the polynomial ring $R[T]$.

\begin{hypothesis}
Throughout \S\ref{sec:polynomial}, let 
$R$ be a ring equipped with a submultiplicative norm $|\cdot|$,
and equip $R[T]$ with the Gauss norm.
\end{hypothesis}

\begin{theorem} \label{T:lift1}
For $\alpha \in \calM(R)$, let $\lambda(\alpha) \in \calM(R[T])$
be the Gauss extension of $\alpha$.
For $\beta \in \calM(R[T])$, let $\mu(\beta) 
\in \calM(R)$ be the
restriction of $\beta$ along $R \to R[T]$. 
\begin{enumerate}
 \item[(a)] The maps  $\lambda$ and $\mu$ are strongly continuous and monotonic.
\item[(b)] For all $\alpha \in \calM(R)$, $(\mu \circ \lambda)(\alpha) = \alpha$.
\item[(c)] For all $\beta \in \calM(R[T])$, $(\lambda \circ \mu)(\beta) \geq \beta$.
\end{enumerate}
\end{theorem}
\begin{proof}
The map $\mu$ is defined as a restriction, and hence is 
strongly continuous.
For $f = \sum_{i=0}^m f_i T^i \in R[T]$ and $\epsilon > 0$,
we have
\begin{align*}
\{\alpha \in \calM(R): \lambda(\alpha)(f) > \epsilon\} &= \bigcup_{i=0}^{m-1}
\{\alpha \in \calM(R): \alpha(f_i) > \epsilon \} \\
\{\alpha \in \calM(R): \lambda(\alpha)(f) < \epsilon\} &= \bigcap_{i=0}^{m-1}
\{\alpha \in \calM(R): \alpha(f_i) < \epsilon \},
\end{align*}
so $\lambda$ is continuous. Similarly, the inverse image of a Weierstrass
(resp.\ Laurent) subspace of $\calM(R[T])$ is a finite union of
Weierstrass (resp.\ Laurent)
subspaces of $\calM(R)$.
Now let 
\[
U = \{\beta \in \calM(R[T]): \beta(f_i) \leq p_i \beta(g) \quad (i=1,\dots,n)\}
\]
be a rational subspace of $\calM(R[T])$ for some
$f_1,\dots,f_n,g \in R[T]$ generating the unit ideal 
in the completion of $R[T]$ and some $p_1,\dots,p_n > 0$.
Write $f_i = \sum_{j=0}^m f_{ij} T^j$ and $g = \sum_{j=0}^m g_j T^j$; then
the $f_{ij}$ and $g_j$ together must generate the unit ideal
(in fact only the $f_{i0}$ and $g_0$ are needed).
We may write
\begin{align*}
\lambda^{-1}(U) &= 
\{\alpha \in \calM(R): 
 \max_{i,j} \{\alpha(f_{ij})/p_i\} \leq \max_j \{\alpha(g_j)\}\} \\
&= \bigcup_{l=0}^m \{\alpha \in \calM(R): 
\alpha(f_{ij}) \leq p_i \alpha(g_l),
\,
\alpha(g_j) \leq \alpha(g_l) \quad
(i=1,\dots,n; j=0,\dots,m)\},
\end{align*}
which is a finite union of rational subspaces of $\calM(R)$.
Since monotonicity is evident, this yields (a).

Of the remaining assertions, (b) is trivial, while (c) holds because
$(\lambda \circ \mu)(\beta)(f_i T^i) \geq \beta(f_i T^i)$ for any 
$f_i \in R$ and any nonnegative integer $i$.
\end{proof}

The following construction is described by Berkovich \cite[Remark~6.1.3(ii)]{berkovich1}.
\begin{lemma} \label{L:homotopy1}
For any $\beta \in \calM(R[T])$ and any $t \in [0,1]$, the function
$H(\beta,t): R[T] \to [0,+\infty)$ defined by 
\begin{equation} \label{eq:continuity1}
 H(\beta,t)(f) = \max_i \left\{ t^i \beta \left( \frac{1}{i!} \frac{d^i}{dT^i}(f) \right)
\right\}
\end{equation}
is a multiplicative seminorm on $R[T]$ dominated by the Gauss norm.
\end{lemma}
\begin{proof}
It is evident that \eqref{eq:continuity1} defines a seminorm dominated
by the Gauss norm. Submultiplicativity follows
from the Leibniz rule in the form
\begin{equation} \label{eq:continuity2}
\frac{1}{i!} \frac{d^i}{dT^i} (gh) = \sum_{j+k=i} \frac{1}{j!} \frac{d^j}{dT^j} (g)
\frac{1}{k!} \frac{d^k}{dT^k} (h).
\end{equation}
To check multiplicativity, we must check that
for $g,h \in R[T]$, we have $H(\beta,t)(gh) \geq H(\beta,t)(g) H(\beta,t)(h)$.
Choose the minimal indices $j,k$ achieving the maxima
in \eqref{eq:continuity1} for $f=g,h$. Then in \eqref{eq:continuity2} for $i = j+k$,
the maximum $\beta$-norm among the summands on the right side of \eqref{eq:continuity2}
is achieved only by the pair $(j,k)$ (as in the proof of Lemma~\ref{L:Gauss multiplicative}).
Since $\beta$ is multiplicative, we obtain
\[
t^i \beta \left( \frac{1}{i!} \frac{d^i}{dT^i} (gh) 
\right) = t^j \beta \left( \frac{1}{j!} \frac{d^j}{dT^j} (g) \right)
t^k \beta \left(
\frac{1}{k!} \frac{d^k}{dT^k} (h) \right),
\]
proving the desired result.
\end{proof}

When $\beta$ is a Gauss seminorm, we can describe $H(\beta,t)$  explicitly.
\begin{lemma} \label{L:homotopy2}
Let $\beta \in \calM(R[T])$ be the $r$-Gauss seminorm 
for the generator $T-x$ for some $x \in R$ and some $r \in [0,1]$.
Then  for $t \in [0,1]$, $H(\beta,t)$ is the $\max\{t,r\}$-Gauss seminorm 
for the generator $T-x$. In particular, $H(\beta,1)$ is the Gauss norm.
\end{lemma}
\begin{proof}
We first check the claim for $t \geq r$.
Let $\gamma$ be the $t$-Gauss seminorm for the generator $T-x$. 
Write $f \in R[T]$ as $\sum_j f_j(T-x)^j$ with $f_j \in R$,
so that $\gamma(f) = \max_j \{t^j |f_j|\}$.
Since $t \geq r$, we have
\[t^i \beta\left( \frac{1}{i!} \frac{d^i}{dT^i}(f) \right)
= t^i \max_{j \geq i} \left\{ \beta\left( \binom{j}{i} f_j (T-x)^{j-i} \right) \right\}
\leq \max_{j \geq i} \{ t^i r^{j-i} |f_j|\} \leq \gamma(f).
\]
It follows that
$\gamma(f) \geq H(\beta,t)(f)$. On the other hand, for each nonnegative integer $i$,
\begin{equation} \label{eq:continuity3}
t^i \beta \left( \frac{1}{i!} \frac{d^i}{dT^i}(f) \right) \geq t^i |f_i|
\end{equation}
because the constant term of $d^i f/dT^i$ is $i! f_i$.
It follows that $H(\beta,t)(f) = \gamma(f)$.

In case $t < r$, we have on one hand $H(\beta,t) \geq \beta$ by taking $i=0$
on the right side of \eqref{eq:continuity1}, and on the other hand $H(\beta,t) \leq H(\beta,
r)
= \beta$ because the right side of \eqref{eq:continuity1} is monotone in $t$.
Hence $H(\beta,t) = \beta$.
\end{proof}

\begin{theorem} \label{T:homotopy1}
The map $H: \calM(R[T]) \times [0,1] \to \calM(R[T])$ is continuous, and 
has the following additional properties.
\begin{enumerate}
 \item[(a)]
For $\beta \in \calM(R[T])$, $H(\beta, 0) = \beta$.
\item[(b)]
For $\beta \in \calM(R[T])$, $H(\beta, 1) = (\lambda \circ \mu)(\beta)$.
\item[(c)]
For $\beta \in \calM(R[T])$ and $t \in [0,1]$, $\mu(H(\beta, t)) = \mu(\beta)$.
\item[(d)]
For $\beta \in \calM(R[T])$ and $s,t \in [0,1]$, $H(H(\beta,s),t) = H(\beta,\max\{s,t\})$.
\end{enumerate}
\end{theorem}
\begin{proof}
The continuity of $H$ is evident from the formula \eqref{eq:continuity1}, since the maximum
on the right side only runs over finitely many terms. 
Of the other properties, (a) and (c) are evident from \eqref{eq:continuity1}.
To check (b), let $\gamma \in \calM(R[T])$ be the Gauss norm.
For $\beta \in \calM(R[T])$, $\beta \leq \gamma$ and so
$H(\beta,1) \leq H(\gamma,1) = \gamma$ by Lemma~\ref{L:homotopy2};
on the other hand, taking $t=1$ in \eqref{eq:continuity3} 
yields $H(\beta,1) \geq \gamma$. (We can also deduce (b) from
Lemma~\ref{L:homotopy2} using Remark~\ref{R:Gauss approach} below.)

To check (d), observe that
\begin{align*}
H(H(\beta,s),t) &= 
\max_j \left\{ t^j \max_k \left\{ s^k \beta\left( \frac{1}{k!} \frac{d^k}{dT^k}
\left( \frac{1}{j!} \frac{d^j}{dT^j}(f) \right) \right) \right\} \right\} \\
&= \max_{j,k} \left\{ t^j s^k \beta \left( \binom{j+k}{j} \frac{1}{(j+k)!}
\frac{d^{j+k}}{dT^{j+k}}(f) \right) \right\} \\
&= \max_{i} \left\{ \beta\left(\frac{1}{i!} \frac{d^{i}}{dT^{i}}(f)
\right)
\max_{j+k=i} \left\{ t^j s^k \beta \left( \binom{j+k}{j} \right) \right\} \right\}.
\end{align*}
Since $\beta$ is a norm, 
$t^j s^k \beta (\binom{j+k}{j}) \leq \max\{s,t\}^i$,
with equality if $s \geq t$ and $(j,k) = (0,i)$, or if $s \leq t$ and $(j,k) = (i,0)$.
This proves (d).
\end{proof}
\begin{cor}
Each subset of $\calM(R)$ has the same homotopy type
as its inverse image under $\mu$.
\end{cor}

\begin{remark}
From Theorem~\ref{T:homotopy1}(b,d), it follows that
for $\alpha \in \calM(R)$ and $t \in [0,1]$, $H(\lambda(\alpha), t) = \lambda(\alpha)$.
This can also be seen more directly: note that $H(\lambda(\alpha),t) \geq \lambda(\alpha)$ from
\eqref{eq:continuity1}, while the reverse inequality follows from 
Theorem~\ref{T:homotopy1}(c) plus
Theorem~\ref{T:lift1}(c). 
\end{remark}

\begin{remark} \label{R:Gauss approach}
One can give an alternate proof of Lemma~\ref{L:homotopy1} using Lemma~\ref{L:homotopy2}, 
as follows. Let $\tilde{\beta} \in 
\calM(R[U][T])$ be the restriction of $\beta$ along
$R[U][T] \to R[T,U]/(U-T) \cong R[T]$.
By Lemma~\ref{L:homotopy2}, $H(\tilde{\beta},t)$ is the $t$-Gauss seminorm
for the generator $T-U$. The restriction of $H(\tilde{\beta},t)$
along $R[T] \to R[U][T]$ is 
$H(\beta,t)$, so the latter is a multiplicative seminorm.

One can go further and take this construction as the definition of 
$H(\beta,t)$, modifying the proof of Theorem~\ref{T:homotopy1} accordingly.
We will not write out the details explicitly, but they will be shadowed in
the context of Witt vectors where no good analogue of
the formula \eqref{eq:continuity1} is available.
(See for instance the proof of Theorem~\ref{T:final homotopy}.)
\end{remark}

\begin{remark} \label{R:contractible}
One may view $\calM(R[T])$ as a closed cylinder of radius 1 over $\calM(R)$,
and $\lambda$ as the section taking each point of $\calM(R)$ to the 
generic point
of its fibre. In this language, Theorem~\ref{T:homotopy1} states that
$\calM(R[T])$ can be uniformly contracted onto the image of $\lambda$;
in particular, each fibre of $\mu$ is contractible.
We may further elucidate the structure of the fibres of $\mu$ by
studying the domination relation; see
Theorem~\ref{T:dominate} and Remark~\ref{R:tree}.
\end{remark}

\begin{defn} \label{D:radius}
For $\beta \in \calM(R[T])$, 
the set of $s \in [0,1]$ for which $H(\beta,s) = \beta$ is nonempty
(because it contains $0$) and closed (by continuity), so it has a greatest
element. This element is called the \emph{radius} of $\beta$,
and is denoted $r(\beta)$; this terminology is justified by the fact
that the $r$-Gauss norm has radius $r$. See also Remark~\ref{R:supremum norm}.
\end{defn}

\begin{theorem} \label{T:dominate}
For $\beta, \gamma \in \calM(R[T])$ satisfying $\mu(\beta) = \mu(\gamma)$
and $\beta \geq \gamma$,
$\beta = H(\gamma,r(\beta))$.
\end{theorem}
\begin{proof}
Put $\alpha = \mu(\beta) = \mu(\gamma)$ 
and $K = \calH(\alpha)$, then identify $\beta, \gamma$
with the corresponding points in $\calM(K[T])$.
These identifications are compatible with the formation of $H(\cdot, t)$;
in particular, they do not change the radius of $\beta$.
It thus suffices to check the case $R = K$, for which we rely on
some analysis of $\calM(K[T])$.
See Lemma~\ref{L:dominate lemma} below.
\end{proof}
\begin{cor} \label{C:dominate}
For $\beta,\gamma \in \calM(R[T])$ satisfying $\mu(\beta) = \mu(\gamma)$
and $\beta \geq \gamma$,
we have  $r(\beta) \geq r(\gamma)$, with equality if and only if
$\beta = \gamma$.
\end{cor}
\begin{proof}
For $t \in [0, r(\gamma)]$, by Theorem~\ref{T:dominate}
and Theorem~\ref{T:homotopy1}(d) we have
\[
H(\beta,t) = H(H(\gamma,r(\beta)),t) = H(H(\gamma,t), r(\beta)) =
H(\gamma,r(\beta)) = \beta,
\]
so $r(\beta) \geq r(\gamma)$. If equality holds, then
$\gamma = H(\gamma,r(\gamma)) = H(\gamma,r(\beta)) = \beta$.
\end{proof}

In order to complete the proof of Theorem~\ref{T:dominate},
we must study $\calM(K[T])$ when $K$ is a complete nonarchimedean field.
In case $K$ is algebraically closed, this was done by
Berkovich \cite[\S 1.4]{berkovich1} 
(see also \cite[Proposition~1.1]{baker-rumely}).
The general case can be found in
\cite[\S 2.2]{kedlaya-semi4}, where it is treated by reduction to
the algebraically closed case. We give here some direct arguments
in terms of the map $H$.

\begin{hypothesis}
For the remainder of \S\ref{sec:polynomial}, 
let $K$ be a field complete for a multiplicative norm $\alpha$,
let $\gotho$ be the valuation subring of a completed
algebraic closure $\CC$ of $K$, and equip both $K[T]$ and
$\CC[T]$ with the Gauss norms.
\end{hypothesis}

\begin{remark} \label{R:quotient}
It is not hard to check that $\calM(K[T])$ is the quotient of
$\calM(\CC[T])$ by the action of the group $\Aut(\CC/K)$ of continuous automorphisms of $\CC$ over $K$; see
\cite[Proposition~1.3.5]{berkovich1}. We will not use this fact 
explicitly, but it is useful to keep in mind.
\end{remark}

\begin{defn}
For $z \in \gotho$ and $r \in [0,1]$, 
let $\tilde{\beta}_{z,r}$ be the $r$-Gauss
norm on $\CC[T]$ for the generator $T-z$,
and let $\beta_{z,r}$ denote the restriction of 
$\tilde{\beta}_{z,r}$ to $K[T]$.
If $z' \in \gotho$ satisfies $\alpha(z' - z) \leq r$,
then $\tilde{\beta}_{z',r} = \tilde{\beta}_{z,r}$;
consequently, if $r > 0$, we always have $\beta_{z,r} = 
\beta_{z',r}$ for some $z' \in \gotho$ which is integral over $K$
(since such $z'$ are dense in $\gotho$).
\end{defn}

\begin{remark} \label{R:supremum norm}
If the norm on $K$ is nontrivial,
then the seminorm $\tilde{\beta}_{z,r}$ can be identified with the supremum
norm over the closed disc in $\CC$ of center $z$ and radius $r$.
Although this fact can be proved directly, it will be convenient for us not to 
deduce it until after making our principal arguments.
See Corollary~\ref{C:disc}.
\end{remark}

\begin{lemma} \label{L:sup norm same center}
For $z \in \gotho$ and $r,s \in [0,1]$, $\beta_{z,r} \geq \beta_{z,s}$
if and only if $r \geq s$.
\end{lemma}
\begin{proof}
If $r \geq s$, then evidently $\beta_{z,r} \geq \beta_{z,s}$.
It remains to show that if $r > s$, then $\beta_{z,r} \neq \beta_{z,s}$.
It suffices to do this when $s > 0$, as when $s = 0$
we can argue that $\beta_{z,r} > \beta_{z,r'} \geq \beta_{z,0}$
for any $r' \in (0,r)$.

Suppose then that $s>0$.
Choose $z' \in \gotho$ integral over $K$ with $\alpha(z-z') \leq s$,
so that
$\beta_{z,r} = \beta_{z',r}$, $\beta_{z,s} = \beta_{z',s}$.
Let $P(T) = \prod_{i=1}^m (T-z_i)$ be the minimal polynomial of $z'$ over $K$;
then $\tilde{\beta}_{z',r}(T-z_i) \geq 
\tilde{\beta}_{z',s}(T-z_i)$ for each $i$,
with strict inequality when $z_i = z'$.
Hence $\beta_{z,r}(P) = \tilde{\beta}_{z',r}(P) > 
\tilde{\beta}_{z',s}(P) = \beta_{z,s}(P)$,
so $\beta_{z,r} \neq \beta_{z,s}$ as desired.
\end{proof}
\begin{cor} \label{C:radius}
For $z \in \gotho$ and $r \in [0,1]$, $r(\beta_{z,r}) = r$.
\end{cor}
\begin{proof}
By Lemma~\ref{L:homotopy2}, we have $H(\beta_{z,r},s) = \beta_{z,\max\{r,s\}}$
for $s \in [0,1]$. The claim then follows from 
Lemma~\ref{L:sup norm same center}.
\end{proof}

\begin{lemma} \label{L:sup norm same radius}
For $z, z' \in \gotho$ and $r \in [0,1]$, the following are equivalent.
\begin{enumerate}
\item[(a)] We have $\beta_{z,r} = \beta_{z',r}$.
\item[(b)] We have $\beta_{z,r} \geq \beta_{z',r}$.
\item[(c)] We have $\beta_{z,r} \geq \beta_{z',0}$.
\item[(d)] There exists $\tau \in \Aut(\CC/K)$ for which 
$\alpha(\tau(z) - z') \leq r$.
\end{enumerate}
\end{lemma}
\begin{proof}
It is clear that (d)$\implies$(a)$\implies$(b)$\implies$(c), 
so it suffices to check that (c)$\implies$(d). For this, we may reduce
to the case $r > 0$ (using the completeness of $\CC$
and the compactness of $\Aut(\CC/K)$).
Assume (c), then choose $y \in \gotho$ integral over $K$ with
$\alpha(y-z) \leq r$, so that $\tilde{\beta}_{y,r} = \tilde{\beta}_{z,r}$.
Let $P(T) = \prod_{i=1}^m (T - y_i)$ be the minimal polynomial of $y$ over $K$,
with the roots ordered so that the sequence $\alpha(y_i - z')$ is nondecreasing.

If (d) fails, then $\alpha(y_i - z') > r$ for $i=1,\dots,m$.
Since $\alpha(y_i - z') \geq \alpha(y_1 - z')$,
we have $\alpha(y_i-z') = \max\{\alpha(y_i-z'), \alpha(y_1-z')\} \geq 
\alpha(y_i - y_1)$.
Hence $\max\{r, \alpha(y_1-y_i)\} \leq \alpha(y_i - z')$
with strict inequality for $i=1$, so
\[
\beta_{z,r}(P) = \beta_{y,r}(P) = \tilde{\beta}_{y_1,r}(P) = \prod_{i=1}^m
\max\{r, \alpha(y_1-y_i)\} <
\prod_{i=1}^m \alpha(z' - y_i)
= \beta_{z',0}(P),
\]
contradiction.
Thus (d) holds, as desired.
(See also \cite[Lemma~2.2.5]{kedlaya-semi4}.)
\end{proof}

The key to the proof of Theorem~\ref{T:dominate} is the following calculation
in the spirit of Remark~\ref{R:Gauss approach}.
\begin{lemma} \label{L:dominate}
For $\beta \in \calM(K[T])$ and $s \in (r(\beta),1]$,
there exists $z \in \gotho$ for which $H(\beta,s) = \beta_{z,s}$.
\end{lemma}
\begin{proof}
Let $S$ be the set of $s \in [0,1]$ for which we can find $z \in \gotho$
(depending on $s$) satisfying $H(\beta,s) = \beta_{z,s}$.
The set $S$ is nonempty because $1 \in S$; it is up-closed because
$H(\beta_{z,r},s) = \beta_{z,\max\{r,s\}}$ by Lemma~\ref{L:homotopy2}
and $H(H(\beta,r),s) = H(\beta,\max\{r,s\})$ by Theorem~\ref{T:homotopy1}(d).
Put $r = \inf S$; to prove the lemma, it suffices to check that 
$r(\beta) \geq r$.

Let $\CC'$ be a completed algebraic closure of $\calH(\beta)$,
fix a continuous embedding of $\CC$ into $\CC'$,
and let $x \in \CC'$ be the image of $T$ under
the map $K[T] \to \calH(\beta)$.
For $s \in [0,1]$, let $\gamma_{x,s}$ denote the $s$-Gauss norm on 
$\calH(\beta)[T]$ for the generator $T-x$,
so that $H(\beta,s)$ is the restriction of $\gamma_{x,s}$ to $K[T]$.
By Lemma~\ref{L:sup norm same radius}
and the stability of $\CC$ under $\Aut(\CC'/\calH(\beta))$, 
for $z \in \gotho$, 
$H(\beta,s) = \beta_{z,s}$ if and only if there exists
$\tau \in \Aut(\CC/K)$ for which
$\alpha(\tau(z) - x) \leq s$. It follows that for $s \in [0, r)$,
$\gamma_{x,s}(T-z) = \max\{s, \alpha(z-x)\}$ is independent of $s$.
Since every element of $K[T]$ factors in $\CC[T]$ as a scalar times a product
of linear polynomials, the restriction of
$\gamma_{x,s}$ to $K[T]$ is constant over $s \in [0,r)$.
Hence $r(\beta) \geq r$, as desired.
\end{proof}
From the proof of Lemma~\ref{L:dominate}, we also read off
the following observation.
\begin{cor} \label{C:infimum}
Suppose that $\beta \in \calM(K[T])$ is such that
$\beta \neq \beta_{z,r}$ for all $z \in \gotho$ and all $r \in [0,1]$.
Then for each $y \in K[T]$, for any sufficiently small $s \in (r(\beta),1]$
(depending on $y$),
$H(\beta,s)(y) = \beta(y)$.
\end{cor}

With this, we may now complete the proof of Theorem~\ref{T:dominate}.
\begin{lemma} \label{L:dominate lemma}
Theorem~\ref{T:dominate} holds for $R=K$.
\end{lemma}
\begin{proof}
If $r(\beta) = 1$, then $\beta = H(\beta,1) = H(\gamma,1)$
by Theorem~\ref{T:homotopy1}(b).
If $r(\gamma) = 1$, then by Theorem~\ref{T:homotopy1}(b) again, 
$\beta \geq \gamma = H(\gamma,1) = H(\beta,1) \geq \beta$
and so $\beta = H(\gamma,1)$.
It is thus safe to assume $r(\beta), r(\gamma) < 1$.
 
For each $s \in (\max\{r(\beta),r(\gamma)\},1]$, 
by Lemma~\ref{L:dominate} we
have $H(\beta,s) = \beta_{z,s}$,
$H(\gamma,s) = \beta_{z',s}$ for some $z,z' \in \gotho$.
Since $\beta \geq \gamma$ implies
$H(\beta,s) \geq H(\gamma,s)$, we have $\beta_{z,s} \geq \beta_{z',s}$,
but by Lemma~\ref{L:sup norm same radius},
this forces $\beta_{z,s} = \beta_{z',s}$.
Hence $H(\beta,s) = H(\gamma,s)$.

If $r(\gamma) > r(\beta)$, by taking the limit as $s \to r(\gamma)^+$,
we deduce that $\gamma = H(\beta,r(\gamma)) > H(\beta,r(\beta)) = \beta$,
contradiction. Hence $r(\beta) \geq r(\gamma)$, and by taking the limit as $s \to r(\beta)^+$,
we deduce $\beta = H(\gamma,r(\beta))$ as desired.
(For an alternate proof, see \cite[Lemma~2.2.12]{kedlaya-semi4}.)
\end{proof}
\begin{cor} \label{C:compatible extension}
For any $\beta, \gamma \in \calM(K[T])$ with $\beta \geq \gamma$,
there exist $\tilde{\beta}, \tilde{\gamma} \in \calM(\CC[T])$ restricting to
$\beta, \gamma$, respectively, for which $\tilde{\beta} \geq \tilde{\gamma}$.
\end{cor}
\begin{proof}
For each finite extension $K'$ of $K$, the map $K[T] \to K'[T]$
is split, so by Lemma~\ref{L:surjective}, the restriction map
$\calM(K'[T]) \to \calM(K[T])$ is surjective. It follows that
$\calM(\CC[T]) \to \calM(K[T])$ is also surjective. 
(See Remark~\ref{R:quotient} for a more precise statement.)

We may thus choose $\tilde{\gamma} \in \calM(\CC[T])$ extending $\gamma$,
then put $\tilde{\beta} = H(\tilde{\gamma}, r(\beta))$. This 
seminorm restricts to $\beta$ by Theorem~\ref{T:dominate}.
(For an alternate proof, see \cite[Lemma~2.2.9]{kedlaya-semi4}.)
\end{proof}

\begin{lemma} \label{L:disc}
Assume that the norm on $K$ is nontrivial.
For $z \in \gotho$ and $r \in [0,1]$, let $D(z,r)$ be the set of
$\beta_{x,0} \in \calM(K[T])$ for which $\beta_{z,r} \geq \beta_{x,0}$.
Then  $D(z,r) = D(z,s)$ if and only if $r = s$.
\end{lemma}
\begin{proof}
It suffices to deduce a contradiction under the assumption that
$D(z,r) = D(z,s)$ for some $r>s>0$.
Pick $z' \in \gotho$ integral over $K$ with
$\alpha(z-z') < s$, so that $D(z,r) = D(z',r)$ and $D(z,s) = D(z',s)$.
Since $D(z,r) = D(z,s)$, for any $\beta_{x,0} \in D(z,r)$,
we have $\beta_{z',s} \geq \beta_{x,0}$ and hence
(by Lemma~\ref{L:sup norm same radius}) $\alpha(\tau(z') - x) \leq s$
for some $\tau \in \Aut(\CC/K)$. Consequently, there are only finitely
many points in $\calM(\CC[T])$ of the form $\tilde{\beta}_{x,s}$ 
which are dominated by $\tilde{\beta}_{z,r}$.

Pick $u \in \gotho$ with $\alpha(u) \in (s,r)$.
For $x,x' \in \gotho$ with $\alpha(z'-x), \alpha(z'-x') \leq \alpha(u)$,
declare $x,x'$ to be equivalent if $\tilde{\beta}_{x,t} =
\tilde{\beta}_{x',t}$ for some $t \in [0,\alpha(u))$.
The resulting equivalence classes may be put in bijection with $\kappa_\CC$
by mapping the class of $x$ to the residue class of $(z'-x)/u$.
Since $\kappa_\CC$ is algebraically closed and hence infinite,
this yields the desired contradiction.
\end{proof}

\begin{cor} \label{C:disc}
Assume that the norm on $K$ is nontrivial.
For $z \in \gotho$ and $r \in [0,1]$, 
$\beta_{z,r} = \sup D(z,r)$.
\end{cor}
\begin{proof}
Put $\gamma_{z,r} = \sup D(z,r)$; it is clear that $\beta_{z,r} \geq 
\gamma_{z,r}$.
By Corollary~\ref{C:radius}, $r(\beta_{z,r}) = r$.
By Theorem~\ref{T:dominate}, $\beta_{z,r} = H(\gamma_{z,r},r)$.

Suppose that $\beta_{z,r} \neq \gamma_{z,r}$;
by Corollary~\ref{C:dominate}, $s = r(\gamma_{z,r})$ must be strictly
less than $r$. 
Pick $s' \in (s,r)$.
By Lemma~\ref{L:dominate},
we can write $H(\gamma_{z,r}, s') = \beta_{z',s'}$ for some $z' \in \gotho$
which is integral over $K$. 
Since $\beta_{z',s'} \geq \gamma_{z,r} \geq \beta_{z,0}$,
by Lemma~\ref{L:sup norm same center}, $\beta_{z',s'} = \beta_{z,s'}$.

By Lemma~\ref{L:disc}, we can find $\beta_{z'',0} \in D(z,r)$
with $\beta_{z'',0} \notin D(z,s') = D(z',s')$.
Hence $H(\gamma_{z,r},s') = \beta_{z',s'} \not\geq \beta_{z'',0}$,
contradicting the fact that $\beta_{z',s'} 
\geq \gamma_{z,r} = \sup D(z,r) \geq \beta_{z'',0}$.
This contradiction forces $\beta_{z,r} = \gamma_{z,r}$, as desired.
\end{proof}

For completeness, we add a classification result formulated
in the style of
Berkovich (see Remark~\ref{R:contractible}).

\begin{theorem} \label{T:classify1}
Each element of $\calM(K[T])$ is of 
exactly one of the following four types.
\begin{enumerate}
\item[(i)]
A point of the form $\beta_{z,0}$ for some $z \in \gotho$.
Such a point has radius $0$ and is minimal.
\item[(ii)]
A point of the form $\beta_{z,r}$ for some $z \in \gotho$
and some $r \in (0,1]$ which is the norm of an element of $\gotho$.
Such a point has radius $r$ and is not minimal.
\item[(iii)]
A point of the form $\beta_{z,r}$ for some $z \in \gotho$
and some $r \in (0,1]$ which is not the norm of an element of $\gotho$.
Such a point has radius $r$ and is not minimal.
\item[(iv)]
The infimum of a decreasing sequence $\beta_{z_i,r_i}$ for which
the sets $D(z_i,r_i)$ have empty intersection.
Such a point has radius $\inf_i \{r_i\} > 0$ and is minimal.
\end{enumerate}
\end{theorem}
\begin{proof}
By Corollary~\ref{C:radius}, 
$r(\beta_{z,r}) = r$. Consequently, types (i), (ii), (iii)
are mutually exclusive.
Moreover, $\beta_{z,r}$ cannot be of type (iv), since
$\beta_{z_i,r_i} \geq \beta_{z,r}$ implies $\beta_{z,0} \in D(z_i,r_i)$.
Consequently, no point can be of more than one type.

It remains to check that any point $\beta \in \calM(K[T])$
not of the form 
$\beta_{z,r}$ is of type (iv) and is minimal of the claimed radius.
Choose a sequence $1 \geq r_1 > r_2 > \cdots$ with infimum $r(\beta)$.
By Lemma~\ref{L:dominate}, for each $i$, 
we have $H(\beta,r_i) = \beta_{z_i,r_i}$ for some
$z_i \in \gotho$.
The sequence $\beta_{z_1,r_1}, \beta_{z_2,r_2}, \dots$ is decreasing
with infimum $\beta$; the sequence $D(z_i, r_i)$ is also decreasing.
For each $z \in \gotho$,
there exists $i$ for which $\beta_{z,r_i} \neq \beta_{z_i, r_i}$;
for such $i$ we have $\beta_{z,0} \notin D(z_i,r_i)$ by
Lemma~\ref{L:sup norm same radius}.
Hence the $D(z_i,r_i)$ have empty intersection; this forces
$\inf_i \{r_i\} > 0$ because $\gotho$ is complete.
Hence $\beta$ is of type (iv); it is minimal by
Theorem~\ref{T:dominate} plus Lemma~\ref{L:dominate}.
Since $\beta = \inf_i \{\beta_{z_i,r_i}\}$ and $r(\beta_{z_i,r_i}) = r_i$
by Corollary~\ref{C:radius}, we have $r(\beta) \geq \inf_i \{r_i\}$;
the reverse inequality also holds because $r_i = r(\beta_{z_i,r_i}) \geq r(\beta)$
by Corollary~\ref{C:dominate}.
\end{proof}

This classification can be used to describe the residual extensions and norm
groups of points in $\calM(K[T])$. For similar results, see
\cite[Lemma~2.2.18]{kedlaya-semi4} or \cite[\S 3]{temkin-uniformization}.
\begin{cor} \label{C:residual}
Let $\beta$ be a point of $\calM(K[T])$,
classified according to Theorem~\ref{T:classify1}.
Let $|\alpha^\times|,|\beta^\times|$ denote the groups of nonzero values assumed
by $\alpha,\beta$, respectively.
\begin{enumerate}
\item[(i)]
For $\beta$ of type (i),
$\kappa_{\calH(\beta)}$ is algebraic over $\kappa_K$,
and $|\beta^\times|/|\alpha^\times|$ is a torsion group.
\item[(ii)]
For $\beta$ of type (ii),
$\kappa_{\calH(\beta)}$ is finitely generated over $\kappa_K$
of transcendence degree $1$,
and $|\beta^\times|/|\alpha^\times|$ is a finite group.
\item[(iii)]
For $\beta$ of type (iii),
$\kappa_{\calH(\beta)}$ is a finite extension of $\kappa_K$,
and $|\beta^\times|/|\alpha^\times|$ is a finitely generated abelian
group of rank $1$.
\item[(iv)]
For $\beta$ of type (iv),
$\kappa_{\calH(\beta)}$ is algebraic over $\kappa_K$,
and $|\beta^\times|/|\alpha^\times|$ is a torsion group.
\end{enumerate}
\end{cor}
\begin{proof}
Recall that for $L/K$ a finite extension of complete nonarchimedean
fields, $\kappa_L$ is a finite extension
of $\kappa_K$ and $|L^\times|/|K^\times|$ is a finite group.
More precisely, by a theorem of Ostrowski
\cite[Theorem~6.2]{ribenboim},
\begin{equation} \label{eq:ostrowski}
\frac{[L:K]}{[\kappa_L:\kappa_K] \#(|L^\times|/|K^\times|)}
\begin{cases}
=1 & (\charac(\kappa_K) = 0) \\
\in \{1,p,p^2,\dots\} & (\charac(\kappa_K) = p>0).
\end{cases}
\end{equation}
Consequently,
in cases (ii) and (iii), it is enough to check the claims after
replacing $K$ by a finite
extension; in cases (i) and (iv), we may replace $K$ by $\CC$ itself. 
We make these assumptions hereafter.

In cases (i), (ii), (iii),
we may now assume that $\beta = \beta_{z,r}$ with
$z \in \gotho_K$. 
In case (i), $\calH(\beta) = K$;
in case (ii), $\kappa_{\calH(\beta)} \cong \kappa_K(x)$
for $x$ the class of $(T-z)/u$ for any $u \in K$ of norm $r$,
and $|\beta^\times|/|\alpha^\times|$ is trivial;
in case (iii), $\kappa_{\calH(\beta)} = \kappa_K$ 
and $|\beta^\times|/|\alpha^\times|$
is free on the generator $r$.

In case (iv), the norm $\alpha$ must be nontrivial.
By Corollary~\ref{C:infimum}, for each $y \in K[T]$,
any sufficiently small $s \in (r(\beta),1]$ satisfies
$H(\beta,s)(y) = \beta(y)$.
If we choose $s \in |\alpha^\times|$, 
we deduce that $|\beta^\times|/|\alpha^\times|$ is trivial.
If we choose $s \notin |\alpha^\times|$, then
for any $z \in K[T]$ with $\beta(z) \leq \beta(y)$, by case (iii),
there must exist $\lambda \in K$ for which $H(\beta,s)(z
- \lambda y) < H(\beta,s)(y)$.
This implies 
\[
\beta(z -\lambda y) \leq H(\beta,s)(z - \lambda y)
< H(\beta,s)(y) = \beta(y),
\]
so $z/y$ and $\lambda$ have the same
image in $\kappa_{\calH(\beta)}$. Hence $\kappa_{\calH(\beta)} = \kappa_K$.
\end{proof}

\begin{remark}
In cases (i) and (iv) of Corollary~\ref{C:residual}, it is not guaranteed that
$\kappa_{\calH(\beta)}$ is finite over $\kappa_K$
or that $|\beta^\times|/|\alpha^\times|$ is a finite group.
We illustrate this with an example of a point of type (i) for which $|\beta^\times|/|\alpha^\times|$ is infinite;
the other claims can be seen by similar arguments.

Let $F$ be an algebraically closed field of characteristic $0$, and take
$K = F((U))$ equipped with the $U$-adic norm (for any normalization).
We may then identify $\CC$ with the completion of the field of Puiseux series in $U$ over $F$.
Inside $\CC$, take $z = \sum_{n=1}^{\infty} U^{n + 1/n!}$ and put $\beta = \beta_{z,0} \in \calM(K[T])$.
We may establish by induction that for each positive integer $m$, $|U|^{1/m!} \in |\beta^\times|$: this is apparent
for $m=1$, and given that this holds for $m-1$, we have $F((U^{1/(m-1)!})) \subseteq \calH(\beta)$
and $\beta(T - \sum_{n=1}^{m-1} U^{n+1/n!}) = |U|^{m+1/m!}$. Consequently,
$|\beta^\times|/|\alpha^\times| \cong \QQ/\ZZ$ is not finite.
\end{remark}

\begin{remark} \label{R:tree}
Theorem~\ref{T:dominate} implies that as a partially ordered set under
domination, $\calM(K[T])$ carries the structure of a tree.
One can capture the tree structure in other ways, for instance,
by exhibiting $\calM(K[T])$ as an inverse limit of finite contractible
simplicial complexes; see for instance \cite[Proposition~1.19]{baker-rumely}.
(This construction is the simplest meaningful case of the main result of \cite{payne}.)

The geometry of $\calM(K[T])$, including the tree interpretation,
has been deployed in a number of apparently unrelated fields.
Here are some representative (but not exhaustive) examples.
\begin{itemize}
 \item
Favre and Jonsson \cite{favre-jonsson, favre-jonsson2, favre-jonsson3}
use the tree structure to study plurisubharmonic singularities of functions of two complex variables.
Some progress has been made in extending to more variables,
by Boucksom, Favre, and Jonsson \cite{boucksom-favre-jonsson}.
\item
Kedlaya \cite{kedlaya-goodformal1, kedlaya-goodformal2} uses the tree structure
to study the local structure of
irregular flat meromorphic connections on algebraic and algebraic varieties.
A related development in $p$-adic cohomology is \cite{kedlaya-semi4}.
\item
Temkin \cite{temkin-uniformization} uses the tree structure to establish 
local uniformization in positive characteristic up to an inseparable
morphism.
\item
Numerous applications have been found in the theory of dynamical systems.
A good starting point for this line of inquiry is the book of Baker
and Rumely \cite{baker-rumely}.
\item
A development closely related to the previous one is the use of
nonarchimedean potential theory in Arakelov theory,
e.g., in the study of equidistribution questions. This is pursued
thoroughly in the work of Chambert-Loir and his collaborators;
see for instance \cite{chambert-loir}.
\end{itemize}
\end{remark}

\section{Witt vectors}

We now introduce the ring of Witt vectors over a perfect ring of characteristic
$p$. These behave a bit like power series in the variable $p$ with coefficients
in the given ring, with the role of the structure morphism
(the injection of the coefficient ring into the series ring)
played by the Teichm\"uller map. The latter map 
is multiplicative but not additive;
nonetheless, we can use it to define raising and lowering
operators $\lambda, \mu$ analogous to the ones from \S\ref{sec:polynomial}.
(We previously considered these operators in \cite{kedlaya-icm}.)

\begin{hypothesis}
For the remainder of the paper, let $R$ denote an $\Fp$-algebra which is \emph{perfect},
i.e., for which the $p$-th power map is a bijection. Unless
otherwise specified, equip $R$ with the trivial norm.
\end{hypothesis}

\begin{remark}
If $R$ is an $\Fp$-algebra which is not necessarily perfect, we can form
the \emph{perfect closure} $R^{\perf}$ 
as the limit of the direct system $R \to R \to \cdots$ in which each arrow 
is the $p$-th power map. We obtain a natural map $R \to R^{\perf}$ by mapping
to the initial term of the direct system; the corresponding map
$\calM(R^{\perf}) \to \calM(R)$ is easily seen to be a homeomorphism.
\end{remark}

\begin{defn}
A \emph{strict $p$-ring} is a (commutative unital) ring $S$
which is $p$-torsion-free and $p$-adically complete and separated, and for which
$S/pS$ is a perfect $\Fp$-algebra.
\end{defn}

\begin{lemma} \label{L:Witt properties}
Let $S$ be a strict $p$-ring with $S/pS \cong R$.
\begin{enumerate}
\item[(a)]
Given $\overline{x} \in R$, let $x_n \in S$ be any lift of
$\overline{x}^{p^{-n}}$. Then the sequence $x_n^{p^n}$ converges $p$-adically
to a limit $[\overline{x}]$ (the \emph{Teichm\"uller lift} of $\overline{x}$),
which is the unique lift of $\overline{x}$
admitting a $p^n$-th root in $S$ for each nonnegative integer $n$.
\item[(c)]
The resulting \emph{Teichm\"uller map} $[\cdot]: R \to S$ is multiplicative.
\item[(d)]
Each $x \in S$ admits a unique representation
$\sum_{i=0}^\infty p^i [\overline{x_i}]$ with $\overline{x_i} \in R$.
\end{enumerate}
\end{lemma}
\begin{proof}
By the binomial theorem, $a \equiv b \pmod{p^m}$ implies
$a^p \equiv b^p \pmod{p^{m+1}}$. Consequently,
\[
x_{m+1}^{p^{m+1}} \cong x_m^{p^m} \pmod{p^{m+1}},
\]
so the $x_m^{p^m}$ converge to a limit $[\overline{x}]$.
Similarly,
for each nonnegative integer $n$, the $x_{m+n}^{p^m}$ converge to a 
$p^n$-th root of $[\overline{x}]$. If 
$x'$ is another lift of $\overline{x}$ admitting a $p^n$-th root $x'_n$
for each nonnegative integer $n$, then
\[
x_{m}^{p^{m}} \cong (x'_m)^{p^m}  = x' \pmod{p^{m+1}},
\]
so $[\overline{x}] = x'$. This proves (a). 

Given (a), the product of two
Teichm\"uller lifts admits a $p^n$-th root for each nonnegative integer $n$,
and so must also
be a Teichm\"uller lift; this yields (b). Since $[\overline{x}]$ is always
a lift of $\overline{x}$, (c) follows.
(See also \cite[\S II.4, Proposition~8]{serre}.)
\end{proof}

\begin{theorem} \label{T:Witt vectors}
There exists a unique (up to unique isomorphism) strict $p$-ring $W(R)$
for which $W(R)/(p) \cong R$. Moreover, the correspondence $R \rightsquigarrow
W(R)$ is covariantly functorial in $R$. 
\end{theorem}
\begin{proof}[Sketch of proof]
For $n=0,1,\dots$, put 
\[
W_n(X_0,\dots,X_n) = \sum_{i=0}^n p^i X_i^{p^{n-i}}.
\]
Given $\Phi \in \ZZ[X,Y]$, there exists a unique sequence
$\phi_0, \phi_1, \dots$ with $\phi_n \in \ZZ[X_0,\dots,X_n,Y_0,\dots,Y_n]$
such that
\begin{equation} \label{eq:Witt polynomials}
W_n(\phi_0, \dots,\phi_n) = \Phi(W_n(X_0,\dots,X_n),W_n(Y_0,\dots,Y_n))
\qquad (n=0,1,\dots);
\end{equation}
using the sequences associated to the polynomials $X-Y, XY$,
we define subtraction and multiplication rules on the set of sequences
$\overline{x_0},\overline{x_1},\dots$ with values in $R$.
This yields a strict $p$-ring $W(R)$ with $W(R)/(p) \cong R$;
more precisely, the sequence $\overline{x_0},\overline{x_1},\dots$ 
corresponds to the element $\sum_{i=0}^\infty p^i [\overline{x_i}^{p^{-i}}]$.
See \cite[\S II.6]{serre} for further details.
\end{proof}

\begin{defn}
The ring $W(R)$ of Theorem~\ref{T:Witt vectors}
is called the
\emph{ring of $p$-typical Witt vectors} with coefficients in $R$;
unless otherwise specified,
we equip $W(R)$ with the $p$-adic norm normalized with $|p| = p^{-1}$.
Since its construction is functorial in $R$, 
$W(R)$ carries an automorphism $\phi$ which corresponds to (and lifts)
the  $p$-power Frobenius map on $R$, called the \emph{Witt vector Frobenius}.
\end{defn}

\begin{remark} \label{R:addition formula}
The addition and multiplication of general elements of $W(R)$ 
is somewhat complicated to express explicitly. 
One important consequence of 
\eqref{eq:Witt polynomials}
is that if we write
$x = \sum_{i=0}^\infty p^i [\overline{x_i}]$,
$y = \sum_{i=0}^\infty p^i [\overline{y_i}]$,
$x-y = \sum_{i=0}^\infty p^i [\overline{z_i}]$, then $\overline{z_i}$
is a polynomial in $\overline{x_j}^{p^{j-i}},\overline{y_j}^{p^{j-i}}$ for $j=0,\dots,i$,
which has integer coefficients, 
is homogeneous of degree 1 for the weighting in which $\overline{x_j},
\overline{y_j}$ have degree 1,
and belongs to the ideal generated by $\overline{x_j}^{p^{j-i}} - \overline{y_j}^{p^{j-i}}$
for $j=0,\dots,i$ (because it vanishes whenever $x=y$).
See also Lemma~\ref{L:Witt explicit} below.
\end{remark}

\begin{lemma} \label{L:Witt explicit}
For $\overline{x} \in R$, write
$[\overline{x}+1] - 1 = \sum_{i=0}^\infty 
p^i [P_i(\overline{x}^{p^{-i}})]$ with $P_i(T) \in \Fp[T]$ as in 
Remark~\ref{R:addition formula}. Then
$P_i(T) \equiv T \pmod{T^2}$.
\end{lemma}
\begin{proof}
Since $[\overline{x}+1]-1$ vanishes when $\overline{x} = 0$,
the polynomial $P_i(T)$ is divisible by $T$. To obtain the
congruence modulo $T^2$, note that
\[
P_i(T) \equiv
p^{-i} \left( (T+1)^{p^i} - 1 - 
\sum_{j=0}^{i-1} p^j P_j(T)^{p^{i-j}}
\right) \pmod{p}.
\]
The coefficient of $T$ on the right side equals 1 (from $p^{-i} (T+1)^{p^i}$)
plus a multiple of $p$ (from all other terms). This proves the claim.
\end{proof}

\begin{remark} \label{R:lift rational}
Suppose $\overline{x_1},\dots,\overline{x_n}$ generate the unit ideal in
$R$. Then $[\overline{x_1}],\dots,[\overline{x_n}]$ generate an ideal
in $W(R)$ containing an element congruent to 1 modulo $p$.
However, any such element is a unit, so the ideal generated is the unit ideal.
\end{remark}

There are two different meaningful types of polynomial extensions of a Witt
ring $W(R)$: the usual polynomial extension of the ring $W(R)$ itself,
and the Witt ring of the perfection of the polynomial extension of the base
rings. These rings enjoy the following relationship.
\begin{lemma} \label{L:Witt split}
Equip $W(R)[T]$ with the Gauss extension of the $p$-adic norm.
\begin{enumerate}
\item[(a)]
The isometric homomorphism
$\psi: W(R)[T] \to W(R[\overline{T}]^{\perf})$
which maps $W(R)$ to 
$W(R[\overline{T}]^{\perf})$
via the functoriality of Witt vectors, and which sends $T$ to $[\overline{T}]$,
is split.
\item[(b)]
The map $\psi^*$ is a quotient map of topological spaces.
\end{enumerate}
\end{lemma}
\begin{proof}
Via $\psi$,
we may identify $W(R[\overline{T}]^{\perf})$ with the $p$-adic completion of
$\cup_{n=1}^\infty W(R)[T^{p^{-n}}]$. Under this identification,
we obtain a splitting by omitting all nonintegral powers of $T$.
Hence $\psi$ is split, yielding (a).

Since $\psi$ is split, $\psi^*$
is surjective by Lemma~\ref{L:surjective}.
Let $U \subseteq \calM(W(R)[T])$ be a subset whose inverse image $V$
in $\calM(W(R[\overline{T}]^{\perf}))$ is open. Let $U', V'$ be the complements
of $U,V$, respectively. Then $V'$ is closed and hence compact because
$\calM(W(R[\overline{T}]^{\perf}))$ is compact. Since $U' = \psi^*(V')$,
$U'$ is quasicompact and hence closed because $\calM(W(R)[T])$ is Hausdorff.
Hence $U$ is open; this proves that $\psi^*$ is a quotient map, yielding (b).
\end{proof}

\begin{remark} \label{R:derivation}
Define the map $\delta: W(R) \to W(R)$ by the formula
\[
\delta(s) = p^{-1} (\phi(s) - s^p) \qquad (s \in W(R)).
\]
The map $\delta$ is an example of a \emph{$p$-derivation} on $W(R)$, in that it 
has the
following properties.
\begin{enumerate}
\item[(a)]
We have $\delta(1) = 0$. (In this example, we also have $\delta([r]) = 0$ for all $r \in R$.)
\item[(b)]
For all $s_1,s_2 \in W(R)$, $\delta(s_1+s_2) = \delta(s_1) + \delta(s_2) - P(s_1,s_2)$,
where the polynomial $P(X,Y) \in \ZZ[X,Y]$ is given by $P(X,Y) = p^{-1}((X+Y)^p - X^p - Y^p)$.
\item[(c)]
For all $s_1, s_2 \in W(R)$, $\delta(s_1s_2) = s_1^p \delta(s_2) + s_2^p \delta(s_1) + p \delta(s_1) \delta(s_2)$.
\end{enumerate}
Such maps were introduced (with a slightly different sign convention)
by Joyal \cite{joyal}, and later exploited heavily by Buium \cite{buium} to transfer some concepts from
the theory of differential equations into arithmetic geometry.
We had hoped to use the $p$-derivation $\delta$ to construct an analogue of the
formula \eqref{eq:continuity1}, but so far we have not found any such analogue.
We will thus be forced to work more indirectly; see
Definition~\ref{D:Gauss norm2} and Theorem~\ref{T:construct Gauss}.
\end{remark}

\section{Raising and lowering for Witt vectors}
\label{sec:raising lowering}

The raising and lowering operators for Witt vectors are defined as follows.
\begin{lemma} \label{L:lambda}
For $\alpha$ a submultiplicative (resp.\ power-multiplicative, multiplicative)
seminorm on $R$ bounded by the trivial norm, the function
$\lambda(\alpha): W(R) \to [0,1]$ given by
\[
\lambda(\alpha) \left( \sum_{i=0}^\infty p^i [\overline{x_i}^{p^{-i}}]
\right) = \max_i \{ p^{-i} \alpha(\overline{x_i})^{p^{-i}} \}
\]
is a submultiplicative (resp.\ power-multiplicative, multiplicative)
seminorm on $W(R)$ bounded by the $p$-adic norm.
\end{lemma}
We will mostly apply this result when $\alpha$ is power-multiplicative, in which case we may use the simpler formula
\[
\lambda(\alpha) \left( \sum_{i=0}^\infty p^i [\overline{x_i}]
\right) = \max_i \{ p^{-i} \alpha(\overline{x_i}) \}.
\]
\begin{proof}
For $\overline{x}, \overline{y} \in R$ and $i$ a nonnegative integer, write $[\overline{x}] - [\overline{y}]
= \sum_{j=0}^\infty p^j [\overline{z_j}^{p^{-j}}]$. By Remark~\ref{R:addition formula},
$\overline{z_j}$ is a polynomial in $\overline{x},\overline{y}$ 
which has integer coefficients and is homogeneous of degree $p^j$. This remains true after taking $p^i$-th powers,
so we may deduce that $\alpha(\overline{z_j}^{p^i}) \leq \max\{\alpha(\overline{x}^{p^i})^{p^j},
\alpha(\overline{y}^{p^i})^{p^{j}}\}$. Consequently,
\begin{align*}
\lambda(\alpha)(p^{i}([\overline{x}]-[\overline{y}])) 
&= \lambda(\alpha)\left( \sum_{j=0}^\infty p^{i+j} [(\overline{z_j}^{p^i})^{p^{-i-j}}] \right) \\
&= \max_j \{p^{-i-j} \alpha(\overline{z_i}^{p^i})^{p^{-i-j}}\} \\
&\leq p^{-i} \max\{\alpha(\overline{x}^{p^i})^{p^{-i}}, \alpha(\overline{y}^{p^i})^{p^{-i}}\} \\
&= \max\{\lambda(\alpha)(p^{i} [\overline{x}]), \lambda(\alpha)(p^{i} [\overline{y}])\}. 
\end{align*}
Similarly, $\lambda(\alpha)(p^{i}([\overline{x}]+[\overline{y}])) 
\leq \max\{\lambda(\alpha)(p^{i} [\overline{x}]), \lambda(\alpha)(p^{i} [\overline{y}])\}$. 

We next establish that $\lambda(\alpha)$ is a seminorm. 
Let $x = \sum_{i=0}^\infty p^i [\overline{x_i}]$, $y = \sum_{i=0}^\infty p^i [\overline{y_i}]$ be two general
elements of $W(R)$, and write $x-y = \sum_{i=0}^\infty p^i [\overline{z_i}]$. For each nonnegative
integer $n$, we will exhibit an equality of the form
\begin{equation}\label{eq:seminorm sum1}
x-y = \left(\sum_{i=0}^{n-1} p^i [\overline{z}_i]\right) \pm p^n w_1 \pm \cdots \pm p^n w_k
\end{equation}
for some nonnegative integer $k = k(n)$ and some $w_1,\dots,w_k \in W(R)$ with the property
that
\begin{equation} \label{eq:seminorm sum2}
\lambda(\alpha)\left(\sum_{i=0}^{n-1} p^i [\overline{z}_i]\right), \lambda(\alpha)(p^n w_1),\dots,\lambda(\alpha)(p^n w_k)
\leq \max\{\lambda(\alpha)(x), \lambda(\alpha)(y)\}.
\end{equation}
From this, it will follow at once that $\lambda(\alpha)(x-y) \leq \max\{\lambda(\alpha)(x), \lambda(\alpha)(y)\}$,
and hence that $\lambda(\alpha)$ is a seminorm.

Suppose that for some nonnegative integer $n$, we are given $w_1,\dots,w_k$ satisfying
\eqref{eq:seminorm sum1} and \eqref{eq:seminorm sum2}. Then condition \eqref{eq:seminorm sum2} is preserved by
modifying \eqref{eq:seminorm sum1} in the following ways.
\begin{enumerate}
\item[1.]
Given a term $\pm p^n w_j$ in \eqref{eq:seminorm sum1}, replace it with the sum of two terms, one of which is
$\pm p^n$ times a Teichm\"uller element.
\item[2.] 
Given two terms of the form $\pm p^n [\overline{w}]$ and $\pm p^n [\overline{w}']$, replace them with their sum.
This maintains \eqref{eq:seminorm sum2} by our earlier argument.
\end{enumerate}
Moreover, the number of summands in \eqref{eq:seminorm sum1} among $\pm p^n w_1,\dots, \pm p^n w_k$
which are not divisible by $p^{n+1}$ never increases, and in fact always decreases in step 2 unless one 
of the two terms is itself divisible by $p^{n+1}$. Consequently, using these operations, we can 
always arrive at the situation where one of the summands in \eqref{eq:seminorm sum1} among $\pm p^n w_1 ,\dots, \pm p^n w_k$  equals 
$p^n [\overline{z}_n]$ and the others are divisible by $p^{n+1}$. This yields a sum of the desired form
with $n$ replaced by $n+1$, completing the proof that $\lambda(\alpha)$ is a seminorm. 
This in turn implies that
\[
\lambda(\alpha)(xy) 
\leq \max_{i,j} \{\lambda(\alpha)(p^i[\overline{x_i}] p^j [\overline{y_j}])\}
\leq \lambda(\alpha)(x) \lambda(\alpha)(y),
\]
so $\lambda(\alpha)$ is submultiplicative.

Suppose now that $\alpha$ is multiplicative. To check that
$\lambda(\alpha)$ is multiplicative, it is enough to check that
$\lambda(\alpha)(xy) \geq \lambda(\alpha)(x) \lambda(\alpha)(y)$ in case
$\lambda(\alpha)(x), \lambda(\alpha)(y) > 0$. Choose the minimal indices
$j,k$ for which $\lambda(\alpha)(p^j [\overline{x_j}])$, 
$\lambda(\alpha)(p^k [\overline{y_k}])$ attain their maximal values. For
\[
x' = \sum_{i=j}^{\infty} p^i [\overline{x_i}],
\qquad
y' = \sum_{i=k}^{\infty} p^i [\overline{y_i}],
\]
on one hand, we may write $x'y' = \sum_{i=j+k}^\infty p^i [\overline{z_i}]$ with
$\overline{z_{j+k}} = \overline{x_j} \overline{y_k}$.
By submultiplicativity, $\lambda(\alpha)(x'y') \geq \lambda(\alpha)(x) \lambda(\alpha)(y) \geq \lambda(\alpha)(x')
\lambda(\alpha)(y') \geq \lambda(\alpha)(x'y')$, so $\lambda(\alpha)(x'y') = \lambda(\alpha)(x) \lambda(\alpha)(y)$.
On the other hand, we have $\lambda(\alpha)(x-x') < \lambda(\alpha)(x)$, $\lambda(\alpha)(y-y') < \lambda(\alpha)(y)$,
so $\lambda(\alpha)(xy-x'y') < \lambda(\alpha)(x) \lambda(\alpha)(y) = \lambda(\alpha)(x'y')$.
Putting everything together,
we deduce that $\lambda(\alpha)$ is multiplicative.
Similarly, if $\alpha$ is power-multiplicative, we see that $\lambda(\alpha)$
is power-multiplicative by taking $x=y$ in the preceding argument.
\end{proof}

\begin{cor} \label{C:lambda}
For $\alpha$ a submultiplicative (resp.\ power-multiplicative, multiplicative)
seminorm on $R$ bounded by the trivial norm, the function
$\Lambda(\alpha): W(R) \to [0,1]$ given by
\[
\Lambda(\alpha) \left( \sum_{i=0}^\infty p^i [\overline{x_i}^{p^{-i}}]
\right) = \sup_i \{ \alpha(\overline{x_i})^{p^{-i}} \}
\]
is a submultiplicative (resp.\ power-multiplicative, multiplicative)
seminorm on $W(R)$ bounded by the trivial norm.
\end{cor}
\begin{proof}
For $x \in W(R)$, we have
\[
\Lambda(\alpha)(x) = \lim_{c \to +\infty} \lambda(\alpha^{c})(x)^{1/c}.
\]
The claims then follow from Lemma~\ref{L:lambda}.
\end{proof}

\begin{remark}
While $\Lambda(\alpha)$ may seem like a more natural analogue of the
Gauss extension than $\lambda(\alpha)$, the proof of the continuity
of $\lambda$ (Theorem~\ref{T:lifting}) does not apply to $\Lambda$;
see Remark~\ref{R:big lambda}.
We thus work primarily with $\lambda$ hereafter.
\end{remark}

\begin{lemma}\label{L:retract}
For $\beta$ a power-multiplicative (resp.\ multiplicative)
seminorm on $W(R)$ bounded by the $p$-adic norm, 
the function $\mu(\beta): R \to [0,1]$
given by
\[
\mu(\beta)(\overline{x}) = \beta([\overline{x}])
\]
is a power-multiplicative (resp.\ multiplicative)
seminorm bounded by the trivial norm.
\end{lemma}
\begin{proof}
Given $\overline{x}, \overline{y} \in R$, choose any $x,y \in W(R)$ lifting them.
For
$(\overline{z}, z) = (\overline{x},x), (\overline{y},y), (\overline{x}+\overline{y}, x+y)$,
for any $\epsilon > 0$, 
for $n$ sufficiently large we have
\[
\max\{\epsilon, \mu(\beta)(\overline{z})\} = \max\{\epsilon,
\beta(\phi^{-n}(z)^{p^n})\}
\]
because $\phi^{-n}(z^{p^n})$ converges $p$-adically to $[\overline{z}]$
by Lemma~\ref{L:Witt properties}(a).
Since $\beta$ is a power-multiplicative
(resp.\ multiplicative) seminorm, we deduce that $\mu(\beta)$
is one as well. (An alternate proof can be obtained using Remark~\ref{R:addition formula}.)
\end{proof}

We now have the following analogue of Theorem~\ref{T:lift1}.
\begin{theorem} \label{T:lifting}
Equip $R$ with the trivial norm and $W(R)$ with the $p$-adic norm.
Define $\lambda: \calM(R) \to \calM(W(R))$, $\mu: \calM(W(R)) \to \calM(R)$
as in Lemma~\ref{L:lambda} and Lemma~\ref{L:retract}.
\begin{enumerate}
 \item[(a)]
The functions $\lambda, \mu$ are strongly continuous and monotonic.
\item[(b)] For all $\alpha \in \calM(R)$, $(\mu \circ \lambda)(\alpha) = \alpha$.
\item[(c)] For all $\beta \in \calM(W(R))$, $(\lambda \circ \mu)(\beta) \geq \beta$.
\end{enumerate}
\end{theorem}
\begin{proof}
For $x = \sum_{i=0}^\infty p^i [\overline{x_i}] \in W(R)$ and $\epsilon > 0$, choose $j > 0$ for which $p^{-j} < \epsilon$; then
$\lambda(\alpha)(p^i [\overline{x_i}]) < \epsilon$ for all $\alpha \in \calM(R)$ and all $i \geq j$. We thus have
\begin{align*}
\{\alpha \in \calM(R): \lambda(\alpha)(x) > \epsilon\} &= \bigcup_{i=0}^{j-1}
\{\alpha \in \calM(R): \alpha(\overline{x_i}) > p^i \epsilon \} \\
\{\alpha \in \calM(R): \lambda(\alpha)(x) < \epsilon\} &= \bigcap_{i=0}^{j-1}
\{\alpha \in \calM(R): \alpha(\overline{x_i}) < p^i \epsilon \},
\end{align*}
so $\lambda$ is continuous. 
Similarly, 
the inverse image of a Weierstrass
(resp.\ Laurent) subspace of $\calM(W(R))$ is a finite union of
Weierstrass (resp.\ Laurent)
subspaces of $\calM(R)$. Now let 
\[
U = \{\beta \in \calM(W(R)): \beta(x_i) \leq q_i \beta(y) \quad (i=1,\dots,n)\}
\]
be a rational subspace of $\calM(W(R))$ for some
$x_1,\dots,x_n,y \in W(R)$ 
generating the unit ideal and some $q_1,\dots,q_n > 0$.
Write $x_i = \sum_{j=0}^\infty p^j [\overline{x_{ij}}]$ and $y = 
\sum_{j=0}^\infty p^j [\overline{y_j}]$; then
the $\overline{x_{ij}}$ and $\overline{y_j}$ 
together must generate the unit ideal
(in fact only the $\overline{x_{i0}}$ and $\overline{y_0}$ are needed).
Moreover, by Remark~\ref{R:rational}, we can choose some nonnegative integer $m$
so that for
$x'_i = \sum_{j=0}^m p^j [\overline{x_{ij}}]$ and $y' = 
\sum_{j=0}^m p^j [\overline{y_j}]$, we also have
\[
U = \{\beta \in \calM(W(R)): \beta(x'_i) \leq q_i \beta(y') \quad (i=1,\dots,n)\}.
\]
We may then write
\begin{align*}
\lambda^{-1}(U) &= 
\{\alpha \in \calM(R): 
 \max_{i,j: j \leq m} \{p^{-j} \alpha(\overline{x_{ij}}) / q_i\} 
\leq \max_{j \leq m} \{p^{-j} \alpha(\overline{y_j})\}\} \\
&= \bigcup_{l=0}^m \{\alpha \in \calM(R): 
p^{-j} \alpha(\overline{x_{ij}}) \leq q_i p^{-l} \alpha(\overline{y_l}),
\,
p^{-j} \alpha(\overline{y_j}) \leq p^{-l} \alpha(\overline{y_l}) \\
&\qquad \qquad (i=1,\dots,n; j=0,\dots, m)\},
\end{align*}
which is a finite union of rational subspaces of $\calM(R)$.
Hence $\lambda$ is strongly continuous.

For $\overline{x} \in R$ and $\epsilon > 0$, we have
\begin{align*}
\{\beta \in \calM(W(R)): \mu(\beta)(\overline{x}) > \epsilon\}
&=
\{\beta \in \calM(W(R)): \beta([\overline{x}]) > \epsilon\} \\
\{\beta \in \calM(W(R)): \mu(\beta)(\overline{x}) < \epsilon\}
&=
\{\beta \in \calM(W(R)): \beta([\overline{x}]) < \epsilon\},
\end{align*}
so $\mu$ is continuous. Similarly, the inverse image of a Weierstrass
(resp.\ Laurent, rational) subspace of $\calM(W(R))$ is a
Weierstrass (resp.\ Laurent, rational)
subspace of $\calM(R)$, using Remark~\ref{R:lift rational} in the rational
case.
Since monotonicity is evident, this yields (a).

The equality (b) is evident from the definitions. The inequality
(c) follows from the definition of $\lambda$ and the observation
that $(\lambda \circ \mu)(\beta)(p^n [\overline{x}]) \geq \beta(p^n [\overline{x}])$ for any 
$\overline{x} \in R$ and any nonnegative integer $n$.
\end{proof}

\begin{remark} \label{R:big lambda}
The proof of continuity of $\lambda$ in Theorem~\ref{T:lifting}
does not apply to $\Lambda$, because we cannot avoid writing
$\{\alpha \in \calM(R): \Lambda(\alpha)(x) < \epsilon\}$
as an \emph{infinite} intersection of open sets.
Similarly, we make no statement (beyond closure) concerning the
inverse image under $\lambda$ of a subspace of
$\calM(W(R))$ of the form $\{\beta \in \calM(W(R)): \beta(x) = 0\}$,
because the inverse image is defined by the 
vanishing of infinitely
many elements of $R$.
\end{remark}

\begin{example}
Here is a simple example to illustrate that $\lambda \circ \mu$ need not be the identity map.
Put $R = \Fp[X]^{\perf}$, so that $W(R)$ is isomorphic to the $p$-adic completion of
$\cup_{n=1}^\infty \Zp[[\overline{X}]^{p^{-n}}]$ (compare Lemma~\ref{L:Witt split}).
The ring $W(R)/([\overline{X}]-p)$ is isomorphic to the completion of
$\cup_{n=1}^\infty \Zp[p^{p^{-n}}]$ for the unique multiplicative extension of the $p$-adic norm; let 
$\beta \in \calM(W(R))$ be the induced seminorm.

Note that $\mu(\beta)(\overline{X}) = \beta([\overline{X}]) = p^{-1}$ and that $\mu(\beta)(y) = 1$ for $y \in \Fp^\times$.
These imply that $\mu(\beta)(y) \leq p^{-p^{-n}}$ whenever $y \in \Fp[\overline{X}^{p^{-n}}]$ is divisible by $\overline{X}^{p^{-n}}$, so
$\mu(\beta)(y) = 1$ whenever $y \in \Fp^\times + \overline{X}^{p^{-n}} \Fp[\overline{X}^{p^{-n}}]$. We 
conclude that $\mu(\beta)$ equals the $\overline{X}$-adic norm on $R$ with the normalization $\mu(\beta)(\overline{X}) = p^{-1}$.
In particular, we have a strict inequality $(\lambda \circ \mu)(\beta) > \beta$.
\end{example}

\begin{remark} \label{R:smaller space}
The corresponding results from \cite{kedlaya-icm} are stated
with $\calM(W(R))$ replaced by the slightly smaller space
$\calM(W(R)[p^{-1}])$, with the arguments unchanged. 
Note however that extending $\lambda(\alpha)$ to $W(R)[p^{-1}]$ requires that
$\lambda(\alpha)(px) = p^{-1} \lambda(\alpha)(x)$, and this holds only if $\alpha$
is power-multiplicative. We will thus mostly restrict to this case in what follows.
This is no serious loss for our purposes, because replacing a seminorm on $R$ with its spectral seminorm
does not change the spectrum.
\end{remark}

\section{Gauss norms}
\label{sec:Gauss norms}

For $\alpha$ a submultiplicative seminorm on $R$ bounded by the trivial norm,
the submultiplicative seminorm $\lambda(\alpha)$ on $W(R)$
behaves like a $(p^{-1})$-Gauss seminorm for the generator $p$.
One would like analogues of Gauss seminorms for other generators, but
unlike in the polynomial case, these cannot be constructed by using 
automorphisms of $W(R)$ to move generators around.
Instead, we use the approach of Remark~\ref{R:construct Gauss by quotient}:
we pass to a polynomial ring equipped with
an appropriate Gauss norm, then return to $W(R)$ by taking a suitable quotient.
The main difficulty in this approach is to transfer multiplicativity
to the quotient norm;
this requires construction of some good coset representatives.

\begin{hypothesis} \label{H:primitive}
Throughout \S\ref{sec:Gauss norms},
equip $R$ with a power-multiplicative seminorm $\alpha$
bounded by the trivial norm, and write $\widehat{R}$ for the separated
completion with respect to $\alpha$.
(The restriction to the power-multiplicative case is made in light of
Remark~\ref{R:smaller space}.)
Choose $\pi = \sum_{i=0}^\infty p^i [\overline{\pi_i}]$
with $\alpha(\overline{\pi_0}) \leq p^{-1}$ and 
$\overline{\pi_1}$ a unit in $R$;
we write $\overline{\pi}$ as shorthand for $\overline{\pi}_0$.
(In the language of \cite{fargues-fontaine}, any such $\pi$
is \emph{primitive of degree $1$}.)
\end{hypothesis}

\begin{defn} \label{D:Gauss norm2}
For $t \in [0,1]$,
define the submultiplicative seminorm $H(\alpha,\pi,t)$ on $W(R)$
as the quotient norm 
on $W(R)[T]/(T - \pi) \cong W(R)$ 
induced by the $(t/p)$-Gauss extension
of $\lambda(\alpha)$ to $W(R)[T]$.
In case $\pi = p - [u]$ for  $u \in R$ with $\alpha(u) \leq p^{-1}$,
we denote $H(\alpha,\pi,t)$ also by $H(\alpha,u,t)$.
\end{defn}

We will show shortly that $H(\alpha,\pi,t)$ is multiplicative
whenever $\alpha$ is (Theorem~\ref{T:construct Gauss}(a)). For this,
we will need some convenient coset representatives for the ideal
$(T-\pi)$ in $W(R)[T]$.

\begin{defn} \label{D:stable}
We say that $x \in W(R)$ is \emph{stable} (or \emph{$\alpha$-stable},
in case we need to specify $\alpha$) if 
$x$ has the form
$\sum_{i=0}^\infty p^i [\overline{x_i}]$ with either
$\alpha(\overline{x_i}) = 0$ for all $i \geq 0$,
or $\alpha(\overline{x_0}) 
> p^{-i} \alpha(\overline{x_i})$ for all $i>0$.
For instance, any Teichm\"uller element is stable.
\end{defn}

\begin{remark} \label{R:why stable}
The term ``stable'' is chosen because of the following fact:
for any stable $x \in W(R)$ with reduction $\overline{x}$
and any $\beta \in \mu^{-1}(\alpha)$,
we have $\beta(x) = \lambda(\alpha)(x) = \alpha(\overline{x})$.
Namely, this is immediate unless $\alpha(\overline{x}) > 0$, in which case
\[
\beta(x - [\overline{x}]) \leq \lambda(\alpha)(x - [\overline{x}]) 
< \lambda(\alpha)([\overline{x}]) =
\alpha(\overline{x}) = \beta([\overline{x}]).
\]
\end{remark}

\begin{lemma} \label{L:stable residue}
Assume $R = \widehat{R}$.
For any $x \in W(R)$, there exists $y = \sum_{i=0}^\infty p^i [\overline{y_i}]
\in W(R)$
with $x \equiv y \pmod{\pi}$
and $\alpha(\overline{y_0}) \geq \alpha(\overline{y_i})$ for all $i>0$.
In particular, $y$ is stable.
\end{lemma}
\begin{proof}
Note that $p^{-1}(\pi - [\overline{\pi}])$ is a unit in $W(R)$; let
$w$ be its inverse.
We construct $x_0,x_1,\ldots \in W(R)$ congruent to $x$ modulo $\pi$,
as follows. Take $x_0 = x$. Given $x_i$,
write $x_i = \sum_{j=0}^\infty p^j [\overline{x_{ij}}]$ with
$\overline{x_{ij}} \in R$, and put
\[
x_{i+1} = x_i - p^{-1} w(x_i - [\overline{x_{i0}}])\pi
= [\overline{x_{i0}}] - p^{-1} w (x_i - [\overline{x_{i0}}]) [\overline{\pi}].
\]
Let $N$ be the least nonnegative integer
for which $\alpha(\overline{x_{N0}}) > \alpha(\overline{\pi})^{N+1}$,
or $\infty$ if no such integer exists.
We check that $\Lambda(\alpha)(x_i) \leq \alpha(\overline{\pi})^i$ for $i \leq N$,
by induction on $i$. The case $i=0$ is immediate. Given the claim for some
$i \leq N$, we have
\[
\Lambda(\alpha)(x_{i+1} - [\overline{x_{i0}}])
\leq \alpha(\overline{\pi}) \Lambda(\alpha)(x_i) \leq \alpha(\overline{\pi})^{i+1}.
\]
If $i < N$, this implies 
$\Lambda(\alpha)(x_{i+1}) \leq \alpha(\overline{\pi})^{i+1}$, completing the induction.
In addition, if $i = N < \infty$, then
$\Lambda(\alpha)(x_{N+1} - [\overline{x_{N0}}]) 
< \Lambda(\alpha)([\overline{x_{N0}}])$
and so $x_{N+1}$ has the desired form. If $N = \infty$, then the series
$\sum_{i=0}^\infty p^{-1} w (x_i - [\overline{x_{i0}}])$ converges
$(p,[\overline{\pi}])$-adically
to a limit $z$ satisfying $x = \pi z$, so we may take $y=0$.
\end{proof}

\begin{defn}
Assume $R = \widehat{R}$.
Then $W(R)$ is $(p,[\overline{\pi}])$-adically complete, so
any sum $\sum_{i=0}^\infty x_i \pi^i$ with $x_i \in W(R)$ 
converges to some limit $x$.
We say that the sequence $x_0,x_1,\dots$ 
forms a \emph{presentation} of $x$ (with respect to $\pi$,
or with respect to $u$ in case $\pi = p - [u]$).
For $x \in W(R)$, $H(\alpha,\pi,t)(x)$ may be computed as the infimum of
\[
 \max_i \{(t/p)^i \lambda(\alpha)(x_i)\}
\]
over all presentations $x_0,x_1,\dots$ of $x$. 

A presentation $x_0,x_1,\dots$ is
\emph{stable} (or \emph{$\alpha$-stable}) if each $x_i$ is stable.
Any $x \in W(R)$ admits a stable presentation; see
Lemma~\ref{L:stable presentation} below. This will imply that the infimum defining
$H(\alpha,\pi,t)(x)$ is always achieved; see Theorem~\ref{T:construct Gauss}(b) below.
\end{defn}

\begin{lemma} \label{L:stable presentation}
If $R = \widehat{R}$,
then every element of $W(R)$ admits a stable presentation.
\end{lemma}
\begin{proof}
Given $x, x_0,\dots,x_{i-1} \in W(R)$,
apply Lemma~\ref{L:stable residue} to construct
a stable $x_i$ congruent to $(x - \sum_{j=0}^{i-1} x_j \pi^j)/\pi^i$ 
modulo $\pi$. This process yields a stable presentation
$x_0,x_1,\dots$ of $x$.
\end{proof}
\begin{cor} \label{C:stable presentation}
For each $x \in W(R)$ and each $\epsilon > 0$,
there exist a nonnegative integer $j$ and 
some stable elements $x_0,\dots,x_j \in W(R)$ such that
\[
\lambda(\alpha)\left(x - \sum_{i=0}^j x_i \pi^i \right) < \epsilon.
\]
\end{cor}
\begin{proof}
Apply Lemma~\ref{L:stable presentation} to construct a stable
presentation $y_0,y_1,\dots$ of $x$ in $W(\widehat{R})$. 
Choose $j$ with $p^{-j-1} < \epsilon$,
then for each $i \in \{0,\dots,j\}$, choose $x_j \in W(R)$
with $\lambda(\alpha)(y_j - x_j) < \epsilon$.
\end{proof}

\begin{remark} \label{R:no Teichmuller}
It is unclear whether one can improve Lemma~\ref{L:stable presentation}
to achieve a presentation using only Teichm\"uller elements,
rather than arbitrary stable elements.
We suspect that this cannot be done, for reasons similar to those given in the
erratum to \cite{kedlaya-revisited}.
\end{remark}

\begin{lemma} \label{L:unique index}
Assume $R = \widehat{R}$.
Let $x_0,x_1,\dots,y_0,y_1,\dots$ be presentations of some $x,y \in W(R)$ for which $xy \neq 0$.
Then for all but finitely many $t \in [0,1]$, there exists a unique pair of indices
$j,k$ maximizing $(t/p)^{j+k} \lambda(\alpha)(x_j y_k)$.
\end{lemma}
\begin{proof}
Since $xy \neq 0$, there must exist some indices $h,i$ for which $x_h y_i \neq 0$.
Then for $t \in (0,1]$, the maximum of  $(t/p)^{j+k} \lambda(\alpha)(x_j y_k)$
can only be achieved by pairs $(j,k)$ for which either
$j+k \leq h+i$ or $p^{h+i-j-k} \geq \lambda(\alpha)(x_h y_i)$.
This limits $(j,k)$ to a finite set independent of $t$; for any two pairs in that set, there is at most one value of $t$
for which both pairs of indices achieve the maximum. By excluding each such value,
we obtain the desired result.
\end{proof}

\begin{theorem} \label{T:construct Gauss}
Choose $t \in [0,1]$, and assume that $\alpha$ is power-multiplicative
(resp.\ multiplicative).
\begin{enumerate}
\item[(a)]
The function $H(\alpha,\pi,t)$ is a power-multiplicative
(resp. multiplicative) seminorm on $W(R)$ bounded by $\lambda(\alpha)$.
\item[(b)]
Assume that $R = \widehat{R}$.
For any stable presentation $x_0,x_1,\dots$ of $x \in W(R)$,
$H(\alpha,\pi,t)(x) = \max_i \{(t/p)^i \lambda(\alpha)(x_i)\}$.
\item[(c)]
For $t \in [0, p\alpha(\overline{\pi})]$ with $p\alpha(\overline{\pi}) > 0$
and $c \in [1, 1 - \log_p (p\alpha(\overline{\pi}))]$,
$H(\alpha,\pi,t) = H(\alpha^{1/c}, \pi, p(t/p)^{1/c})^c$.
\item[(d)]
For $t \in [p\alpha(\overline{\pi}),1]$ with $t>0$, 
we have
$H(\alpha,\pi,t) = \lambda(\alpha^{1/c})^c$
for $c = 1 - \log_p t$. 
In particular, $H(\alpha, \pi,1) = \lambda(\alpha)$.
(For $t=0$, we obtain the same conclusion by interpreting
$\lambda(\alpha^{1/c})^c$ for $c = +\infty$ as the
restriction of $\alpha$ along $W(R) \to R$.)
\end{enumerate}
\end{theorem}
\begin{proof}
We may assume throughout that $R = \widehat{R}$.
Given $x,y \in W(R)$,
apply Lemma~\ref{L:stable presentation} to construct stable 
presentations $x_0,x_1,\dots$,
$y_0,y_1,\dots$ of $x,y$. We verify that
\begin{equation} \label{eq:compare stable}
H(\alpha, \pi,t)(xy) \geq \max_{j+k} \{(t/p)^{j+k} \lambda(\alpha)(x_j y_k)\} \qquad (t \in [0,1]).
\end{equation}
Suppose the contrary; then $xy \neq 0$. We must  have a presentation
$z_0,z_1,\dots$ of $xy$ for which
\begin{equation} \label{eq:compare stable2}
\max_i \{(t/p)^i \lambda(\alpha)(z_i)\} < \max_{j,k}\{ (t/p)^{j+k} \lambda(\alpha)(x_j y_k)\}
\end{equation}
for some $t \in [0,1]$. 
Let $S$ be the set of $t \in (0,1]$ for which there are unique indices $j,k$ maximizing
$(t/p)^{j+k} \lambda(\alpha)(x_j y_k)$.
By Lemma~\ref{L:unique index},
the complement of $S$ in $[0,1]$ is finite. 
Since \eqref{eq:compare stable2} holds for some $t$
and both sides of \eqref{eq:compare stable2} are
continuous in $t$, \eqref{eq:compare stable2} must hold for some $t \in S$.
Choose some such $t$ and put $s = (\log p)/(\log (p/t))$, so that $(t/p)^{is} = p^{-i}$. 
We then have
\[
x_j y_k \pi^{j+k} = \sum_{i=0}^\infty z_i \pi^i 
- \sum_{(j',k') \neq (j,k)} x_{j'} y_{k'} \pi^{j'+k'}
\]
but
\[
\lambda(\alpha^s)(x_j y_k \pi^{j+k}) > \max_i \{\lambda(\alpha^s)(z_i \pi^i)\}, 
\max_{(j',k') \neq (j,k)}
\{\lambda(\alpha^s)(x_{j'} y_{k'} \pi^{j'+k'})\}.
\]
This gives a contradiction, and \eqref{eq:compare stable} follows.

To deduce (a), 
note that from the definition, $H(\alpha, \pi,t)$ is evidently
a submultiplicative seminorm bounded by $\lambda(\alpha)$.
If $\alpha$ is multiplicative, then $H(\alpha,\pi,t)$ is multiplicative
because \eqref{eq:compare stable} implies
$H(\alpha,\pi, t)(xy) \geq H(\alpha, \pi,t)(x) H(\alpha,\pi, t)(y)$.
Similarly, if $\alpha$ is power-multiplicative, then so is $H(\alpha,\pi,t)$.
To deduce (b), 
apply \eqref{eq:compare stable} with $y = y_0 = 1$ and $y_i = 0$ for $i>0$.

Suppose $t \in [0, p\alpha(\overline{\pi})]$ with $p\alpha(\overline{\pi}) > 0$
and $c \in [1, 1 - \log_p (p\alpha(\overline{\pi}))]$.
Since $c \leq 1 - \log_p (p \alpha(\overline{\pi}))$, we have
$\alpha^{1/c}(\overline{\pi}) \leq p^{-1}$, so 
$H(\alpha^{1/c},\pi,p(t/p)^{1/c})$ is well-defined.
Since $c \geq 1$, any $\alpha$-stable element is also $\alpha^{1/c}$-stable (as in Remark~\ref{R:why stable}),
so we may apply (b) to deduce (c).

To deduce (d), note that
$H(\alpha,\pi,1) \leq \lambda(\alpha)$ from the definition of $H(\alpha,\pi,t)$ as a quotient norm,
whereas $H(\alpha,\pi,1) \geq \lambda(\alpha)$ from (b).
Given this, if $t \in [p\alpha(\overline{\pi}),1]$ with $t>0$
and $c = 1 - \log_p t$, then in particular $c \in [1, 1 - \log_p (p \alpha(\overline{\pi}))]$, so
(c) implies $H(\alpha,\pi,t) = H(\alpha^{1/c}, \pi, 1)^c = \lambda(\alpha^{1/c})^c$. This yields (d) for $t>0$;
the case $t=0$ follows by continuity.
\end{proof}

\begin{cor} \label{C:convex}
For any $x \in W(R)$ with $\lambda(\alpha)(x) \neq 0$,
the function $v_x(r) = -\log H(\alpha,\pi,e^{-r})(x)$ on $[0, +\infty)$
is continuous, concave, nondecreasing, and piecewise affine with nonnegative integer slopes.
\end{cor}
\begin{proof}
This is apparent from Theorem~\ref{T:construct Gauss}(b)
and the existence of stable presentations in case
$R = \widehat{R}$ (Lemma~\ref{L:stable presentation}).
\end{proof}

As an application of Corollary~\ref{C:convex}, we exhibit a computation
which is not straightforward using
stable presentations.
\begin{lemma} \label{L:compute linear}
For $u,u' \in R$ with $\alpha(u),\alpha(u') \leq p^{-1}$ and $t \in [0,1]$, 
\[
H(\alpha, u,t)(p-[u']) = \max\{t/p, H(\alpha,u,0)(p-[u'])\}.
\]
\end{lemma}
\begin{proof}
Consider the 
functions
\begin{align*}
f(r) &= -\log H(\alpha,u,e^{-r})(p-[u']) \\
g(r) &= -\log \max\{e^{-r}/p, H(\alpha,u,0)(p-[u'])\}.
\end{align*}
Note that 
$f$ and $g$ take the same value $\log p$ at $r=0$, and 
tend to the same (possibly infinite) limit as $r \to \infty$.
In case $\alpha(u - u') = p^{-1}$, then $[u] - [u']$
is stable, so $[u] - [u'], 1, 0, 0, \dots$ 
is a stable presentation of $p - [u']$ with respect to $u$.
By Theorem~\ref{T:construct Gauss}(b),
$H(\alpha,u,t)(p - [u']) = p^{-1}$ for all $t \in [0,1]$,
so $f=g$.

In case $\alpha(u - u') < p^{-1}$, we have
$H(\alpha,u,1)([u] - [u']) = \lambda(\alpha)([u] - [u']) < p^{-1}$.
Consequently, we have $H(\alpha,u,1)(p - [u']) = t/p$ for 
$t$ close to 1. This means that in a right neighborhood of $r=0$,
$f(r)$ and $g(r)$ are both affine with slope 1. 
By Corollary~\ref{C:convex}, both functions are continuous, concave,
nondecreasing, and piecewise affine with nonnegative
integer slopes; hence each function either
persists with slope 1 forever, or becomes constant after some point.
Given this information plus the fact that
$f$ and $g$ have the same limiting value, the two functions
are forced to coincide.
\end{proof}

\begin{remark} \label{R:radius 0}
Note that $H(\alpha,\pi,0)$ is the quotient norm on $W(R)/(\pi)$
induced by $\lambda(\alpha)$. In particular, if $\alpha$ is a multiplicative
norm, then $H(\alpha,\pi,0)(x) = 0$ if and only if $x$ is divisible by $\pi$.

Note also that any $\beta \in \calM(W(R))$ with $\mu(\beta) = \alpha$
and $\beta(\pi) = 0$ must equal $H(\alpha,\pi,0)$. 
Namely, it suffices to check this assuming that $R = \widehat{R}$.
Given $x \in W(R)$,
apply Lemma~\ref{L:stable presentation} to construct
a stable presentation $x_0,x_1,\dots$ of $x$. 
By Theorem~\ref{T:construct Gauss}(b) and 
Remark~\ref{R:why stable}, $H(\alpha,\pi,0)(x) = \beta(x_0) = \beta(x)$.
\end{remark}

\begin{remark} \label{R:radius 0b}
One consequence of Remark~\ref{R:radius 0} is that if
$\alpha$ is a multiplicative norm and $u, u' \in R$ are
such that $H(\alpha,u,0)(p - [u']) = 0$, then 
$p - [u'] = y(p-[u])$ for some unit $y \in W(R)$.
This implies that $H(\alpha,u,t) = H(\alpha,u',t)$ for all $t \in [0,1]$;
it does not imply $u = u'$ (see Example~\ref{exa:same norm}), but it does limit the possibilities
for $u'$, as in Remark~\ref{R:same norm} below.
\end{remark}

We will need the following variant of Lemma~\ref{L:stable residue}. Thanks to Peter Scholze for pointing out
a mistake in a previous version of this lemma, and to Liang Xiao for suggesting this replacement.
\begin{lemma} \label{L:stable approximation}
For any $x \in W(R)$ and any $\epsilon > 0$, there exists $y = \sum_{i=0}^\infty p^i [\overline{y_i}] \in W(R)$
with $x \equiv y \pmod{\pi}$ and $\beta(\overline{y_i}) \leq \max\{\beta(\overline{y_0}), \epsilon\}$
for all $i>0$ and all $\beta \in \calM(R)$.
\end{lemma}
\begin{proof}
Define $x = x_0, x_1, \dots$ as in the proof of Lemma~\ref{L:stable residue},
and again write $x_i = \sum_{j=0}^\infty p^j [\overline{x_{ij}}]$.
Take $n$ to be a nonnegative integer for which $\alpha(\overline{\pi})^n \leq \epsilon$,
and put $y = x_n$. For $\beta \in \calM(R)$, let $N$ be the least nonnegative integer for which
$\beta(\overline{x_{N0}}) > \beta(\overline{\pi})^{N+1}$, or $\infty$ if no such integer exists.
By arguing as in the proof of Lemma~\ref{L:stable residue},
we see that if $n \leq N$, then
$\Lambda(\beta)(y) \leq \beta(\overline{\pi})^n \leq \alpha(\overline{\pi})^n \leq \epsilon$;
if instead $n > N$, then $\beta(\overline{x_{i0}}) = \Lambda(\beta)(\overline{x_{i+1}}) = \beta(\overline{x_{N0}})$
for $i = N$, but also for $i > N$ by induction on $i$.
\end{proof}

\section{Newton polygons and factorizations}
\label{sec:Newton polygons}

The development of the basic algebra of polynomials over a complete
nonarchimedean field is often phrased in the language of Newton polygons.
One can develop a similar device to deal with the ring of Witt vectors over
a perfect valuation ring; we use these to develop an analogue 
of the factorization of a polynomial over an algebraically closed field
into linear constituents. This observation is due to Fargues
and Fontaine; see Remark~\ref{R:fargues-fontaine}.

\begin{hypothesis} \label{H:Newton polygons}
Throughout \S\ref{sec:Newton polygons}, 
let $\gotho$ be the valuation ring of a perfect field
of characteristic $p$ complete under a multiplicative norm $\alpha$.
Equip $W(\gotho)$ with the norm $\lambda(\alpha)$,
which is also multiplicative by Lemma~\ref{L:lambda}.
\end{hypothesis}

\begin{defn}
Let $W^\dagger(\Frac \gotho)$ denote the set of 
$x = \sum_{i=0}^\infty p^i [\overline{x_i}] \in W(\Frac \gotho)$
for which $p^{-i} \alpha(\overline{x_i}) \to 0$ as $i \to \infty$.
For $T$ the set of nonzero Teichm\"uller lifts in
$W(\gotho)$, we may identify $W^\dagger(\Frac \gotho)$ with
the completion of $T^{-1} W(\gotho)$ for the unique multiplicative extension
of $\lambda(\alpha)$. 
We define \emph{stable} elements of $W^\dagger(\Frac \gotho)$
using the same definition as in $W(\gotho)$ (see Definition~\ref{D:stable}).

For $x = \sum_{i=0}^\infty p^i [\overline{x_i}] \in W^\dagger(\Frac \gotho)$ nonzero, 
the \emph{Newton polygon} of $x$ is the lower boundary of the convex hull of the set
\[
\bigcup_{i=0}^\infty \{(a,b) \in \RR^2: a
\geq -\log \alpha(\overline{x_i}), \, b \geq i\},
\]
minus any segments of slope less than $(\log p)^{-1}$.
The \emph{multiplicity} of $r \in [-(\log p)^{-1},0)$ 
in the Newton polygon of $x$ is the
height of the segment of the Newton polygon of slope $r$,
or $0$ if no such segment exists.
\end{defn}

\begin{lemma} \label{L:np}
For $x,y \in W^\dagger(\Frac \gotho)$ nonzero and 
$r \in [-(\log p)^{-1},0)$, the multiplicity of $r$ in the
Newton polygon of $xy$ is the sum of the multiplicities of $r$ in the
Newton polygon of $x$ and $y$.
\end{lemma}
\begin{proof}
The argument is similar to the proofs that $\lambda(\alpha)$ and 
$H(\alpha,u,t)$ are multiplicative
(Lemma~\ref{L:lambda} and Theorem~\ref{T:construct Gauss}(a)), so we omit the details.
See also \cite[Lemma~2.1.7]{kedlaya-revisited}.
\end{proof}
\begin{cor} \label{C:stable units}
The units in $W^\dagger(\Frac \gotho)$ are precisely the nonzero stable elements,
which are in turn the elements with no slopes in their Newton polygons.
\end{cor}
\begin{proof}
Any nonzero element $x$ of $W^\dagger(\Frac \gotho)$
can be written uniquely as $[\overline{y}]z$ with $\overline{y} \in \gotho$ nonzero
and $z \in 1 + p W^\dagger(\Frac \gotho)$.
If $x$ is stable, then $\lambda(\alpha)(z-1) < 1$, so $z$ is a unit,
as then is $x$. Conversely, if $x$ is a unit, then by
Lemma~\ref{L:np} the multiplicity of each $r \in [-(\log p)^{-1},0)$ in the Newton
polygon of $x$ is zero, so $x$ must be stable.
\end{proof}
\begin{lemma} \label{L:prime factors}
For $u \in \gotho$ with $\alpha(u) \leq p^{-1}$,
the ideal $(p-[u])$ in $W(\gotho)$ is prime.
\end{lemma}
\begin{proof}
If $xy$ is divisible by $p-[u]$, then $H(\alpha,u,0)(xy) = 0$.
Since $H(\alpha,u,0)$ is multiplicative by Theorem~\ref{T:construct Gauss}(a),
this forces either $H(\alpha,u,0)(x) = 0$ or $H(\alpha,u,0)(y) = 0$.
Without loss of generality, suppose $H(\alpha,u,0)(x) = 0$;
then by Remark~\ref{R:radius 0}, $x$ is divisible by $p-[u]$.
\end{proof}

\begin{lemma} \label{L:get factor}
Suppose that $x \in W(\gotho)$ is nonzero and not stable.
\begin{enumerate}
\item[(a)]
There exists
an isometric embedding $\gotho \to \gotho_0$ of complete perfect valuation rings
of characteristic $p$ such that in $W(\gotho_0)$,
$x$ is divisible by $p-[u]$ for some $u \in \gotho_0$ with $\alpha(u) \leq p^{-1}$.
\item[(b)]
If $\Frac \gotho$ is algebraically closed,
we may take $u \in \gotho$.
\end{enumerate}
\end{lemma}
\begin{proof}
In both cases, we may assume $x$ is not divisible by $p$, as otherwise $u=0$ works.
By rescaling $\alpha$, we may reduce to the case where $-(\log p)^{-1}$
has nonzero multiplicity in the Newton polygon of $x$. In this case, 
to prove (a), we
will construct $u$ so that $\alpha_0(u) = p^{-1}$,
for $\alpha_0$ the extended norm on $\gotho_0$.

Let $S$ be the completion of $W^\dagger(\Frac \gotho)[p^{-1}]$ for the unique
multiplicative extension of $\lambda(\alpha)$. 
Then $\kappa_S$ is a Laurent polynomial ring over 
$\kappa_\gotho$ generated by the class of $x[\overline{x}^{-1}] - 1$.
It follows that $x$ is not a unit in $S$.

Equip $S/(x)$ with the quotient norm. Since $S/(x)$ is nonzero,
$\calM(S/(x)) \neq \emptyset$
by Theorem~\ref{T:Gelfand transform1}, 
Choose $\beta \in \calM(S/(x))$; it corresponds to an element of 
$\mu^{-1}(\alpha)$ with $\beta(p) = p^{-1}$ and $\beta(x) = 0$.
(Note that the condition on $\beta(p)$ would not have been guaranteed
had we used $W^\dagger(\Frac \gotho)$ instead of $S$.)
Restrict $\beta$ to $S$ and then to $W(\gotho)$, then 
use the isomorphism $W(\gotho)[T]/(T-p) \cong W(\gotho)$ to further restrict $\beta$
to $W(\gotho)[T]$. 

Since the restriction
map $\psi^*: \calM(W(\gotho[\overline{T}]^{\perf})) \to \calM(W(\gotho)[T])$
of Lemma~\ref{L:Witt split} is surjective,
we can extend $\beta$ to $\beta_0 \in 
\calM(W(\gotho[\overline{T}]^{\perf}))$.
Put $\alpha_0 = \mu(\beta_0)$, let $\gotho_0$ be the valuation ring
of $\calH(\alpha_0)$, and take
$u$ to be the image of $[\overline{T}]$ in $\gotho_0$.
Since $\beta_0(x) = 0$,
$x$ is divisible by $p-[u]$ in $W(\gotho_0)$ by Remark~\ref{R:radius 0}.
This proves (a).

To prove (b), keep notation as above, but suppose by way of contradiction 
that $u \notin \gotho$. Since $\Frac \gotho$ is algebraically closed, 
the restriction of the norm on $\gotho_0$ to $\gotho[\overline{T}]$ defines a point of
$\calM(\gotho[\overline{T}])$ whose radius $r$ is positive. 
This in turn implies if we equip $\gotho_0[\overline{T}]^{\perf}$ with the
$r$-Gauss norm for the generator $\overline{T} - u$, then the map
$\gotho_0[\overline{T}]^{\perf} \to \gotho_0 \widehat{\otimes} \gotho_0$ taking $\gotho$ to $\gotho \otimes 1$
and $\overline{T}$ to $1 \otimes u$ is isometric.

Choose a nonnegative integer $n$ for which $p^{-p^n-1} < r$. 
For $i=0,\dots,p^n$, let $\gotho_i$ be a copy of $\gotho_0$ in which $u_i$ denotes the element corresponding to $u$.
Restrict $\alpha_0$ along the map $\gotho_0 \widehat{\otimes}_{\gotho} \cdots \widehat{\otimes}_{\gotho} \gotho_{p^n} \to \gotho_0$ to obtain a seminorm $\alpha'$, and let $\gotho'$ be the valuation ring of $\calH(\alpha')$.

For $0 \leq i < j \leq p^n$, we have 
$H(\alpha',u_0,0)(p-[u_i]) = H(\alpha',u_0,0)(p-[u_j]) = 0$, so 
$[(u_j/u_i)^{p^{-n}}]$ maps to a $p^n$-th root of unity in $\calH(H(\alpha',u_0,0))$.
If this root were 1, then by Remark~\ref{R:radius 0}, $[(u_j/u_i)^{p^{-n}}] - 1$
would be divisible by $p-[u_0]$ in $W(\gotho')$,
which would imply $\alpha'((u_j/u_i)^{p^{-n}} - 1) \leq p^{-1}$ and $\alpha'(u_j - u_i) \leq p^{-p^{n}} \alpha(u) < r$;
however, this would contradict the description of $\gotho_0 \widehat{\otimes} \gotho_0$ from the first paragraph.
It follows that $[(u_j/u_i)^{p^{-n}}]$ maps to a nontrivial $p^n$-th root of unity in $\calH(H(\alpha',u_0,0))$,
but by the pigeonhole principle, this cannot hold simultaneously for all $i,j$. The resulting contradiction
forces $u \in \gotho$; this yields (b).
\end{proof}

\begin{remark} \label{R:same norm}
By a similar argument to the proof of Lemma~\ref{L:get factor}(b),
one may show the following.
For $u \in \gotho$ with $\alpha(u) \leq p^{-1}$,
for each $\epsilon > 0$, the set of $u' \in \gotho$ with 
$H(\alpha,u,0)(p - [u']) = 0$ is contained in
finitely many residue classes
modulo elements of norm at most $\epsilon$. More precisely, if $p^{-p^n-1} < \epsilon$, there are at most $p^n$
such classes.
\end{remark}

\begin{theorem} \label{T:factor}
Assume that $\Frac \gotho$ is algebraically closed.
For $x \in W(\gotho)$ nonzero and not stable, 
we can write $x = y (p - [u_1]) \cdots (p - [u_n])$
for some nonzero stable $y \in W(\gotho)$ and some $u_1,\dots,u_n
\in \gotho$ with $\alpha(u_1),\dots,\alpha(u_n) < p^{-1}$.
\end{theorem}
\begin{proof}
We may divide out powers of $p$ as needed to reduce to the case
where $x$ is not divisible by $p$.
Let $n$ be the sum of all multiplicities in the Newton polygon of $x$;
this is a nonnegative integer.
We check that for $m=0,\dots,n$, we can find $u_1,\dots,u_m \in \gotho$
such that $x$ is divisible by $(p-[u_1]) \cdots (p-[u_m])$. 
This proceeds by induction on $m$, with empty base case $m=0$.
For the induction step, since $m < n$, by Lemma~\ref{L:np},
the sum of all multiplicities in the Newton polygon of $x_m = 
x/((p-[u_1])\cdots(p-[u_m]))$
is nonzero, so $x_m$ cannot be stable. We may thus apply 
Lemma~\ref{L:get factor} to construct $u_{m+1}$ of the desired form.

Given $u_1,\dots,u_n$ as above, put $y = x/((p-[u_1])\cdots(p-[u_n])) \in W(\gotho_1)$.
By Lemma~\ref{L:np}, the Newton polygon of $y$ has no slopes, so $y$ is stable.
This gives the desired factorization.
\end{proof}

\begin{remark}  \label{R:new Newton poly}
For a fixed choice of $u \in \gotho$ with $\alpha(u) \leq p^{-1}$,
one can also define Newton polygons which keep track of the 
seminorms $H(\alpha,u,t)$, either by examining stable presentations,
or by taking the concave duals of the graphs of the functions $v_r(x)$
from Corollary~\ref{C:convex}.
We leave it to the reader to formulate and verify the multiplicativity
property in this case.

By analogy with the theory of Newton polygons for polynomials over a complete
nonarchimedean field, one may expect that 
for $x \in W(\gotho)$ nonzero, we can use the Newton polygon to read off
some information about the factors occuring in the representation
$x = y(p - [u_1])\cdots(p-[u_n])$ given by Theorem~\ref{T:factor}.
Again, this is equivalent to a statement about the function $v_r(x)$,
which may be deduced from Lemma~\ref{L:compute linear}:
the right slope of $v_r(x)$ at $r$
counts the number of indices
$i$ for which $H(\alpha,u,0)(p - [u_i]) < e^{-r}/p$.
\end{remark}

\begin{remark} \label{R:fargues-fontaine}
A similar analysis of elements of $W(\gotho)$,
including a somewhat more constructive proof of Theorem~\ref{T:factor},
appears in the
development of $p$-adic Hodge theory given by Fargues and
Fontaine \cite{fargues-fontaine}.
\end{remark}

As an application of Theorem~\ref{T:factor}, we can produce an example of distinct 
$u,u' \in \gotho$ with $\alpha(u) = \alpha(u') = p^{-1}$ for which
$p-[u]$ and $p-[u']$ generate the same ideal in $W(\gotho)$, as promised in 
Remark~\ref{R:radius 0b}. This example is crucial in $p$-adic Hodge theory,
as in \cite{kedlaya-icm} or \cite{fargues-fontaine}.

\begin{example} \label{exa:same norm}
Suppose that there exists $\overline{x} \in \gotho$ with $\alpha(\overline{x}) = p^{-p/(p-1)}$.
Put
\[
\pi = \sum_{j=0}^{p-1} [1 + \overline{x}]^{j/p} \in W(\gotho)
\]
and write $\pi = \sum_{i=0}^\infty p^i [\overline{\pi_i}]$. Then
$\alpha(\overline{\pi_0}) = p^{-1}$ and $\alpha(\overline{\pi_1}) = 1$, so 
by Theorem~\ref{T:factor} we can write $\pi = (p-[u])y$ for some $u \in \gotho$
with $\alpha(u) = p^{-1}$ and some unit $y \in W(\gotho)$. 
Note that $\pi(1 - [1 + \overline{x}]^{1/p}) = 1 - [1+\overline{x}]$,
so $1$ and $[1 + \overline{x}]$ have the same image in $W(\gotho)/(p-[u]) = W(\gotho)/(\pi)$.
Consequently, $H(\alpha,u,0)(p-[u']) = 0$ for $u' = u(1 + \overline{x})^\gamma$
for any $\gamma \in \ZZ$, and by continuity also for any $\gamma \in \Zp$ if we use the binomial
series to define $(1+ \overline{x})^\gamma$.
By Remark~\ref{R:radius 0}, $p-[u]$ and $p-[u']$ generate the same ideal in $W(\gotho)$.
\end{example}

\section{Restriction of Gauss norms}
\label{sec:restriction}

We are now ready to construct a strong deformation retract between
the spectra of $R$ and $W(R)$. We cannot directly imitate the construction
for polynomial rings, for lack of an analogue of the formula
\eqref{eq:continuity1} (see Remark~\ref{R:derivation}).
We thus instead follow the approach suggested in Remark~\ref{R:Gauss approach}.
Given an element of $\calM(W(R))$, we express it as the restriction from
a larger Witt ring of a seminorm of the form $H(\alpha,u,0)$,
then define the homotopy by restricting the corresponding seminorms $H(\alpha,u,t)$. Before embarking on this construction, we read off a key
continuity property from the construction of 
the seminorms $H(\alpha,u,t)$.
\begin{theorem} \label{T:continuous1}
Equip $R$ with a power-multiplicative norm $\alpha$ bounded above by the
trivial norm,
equip $W(R)$ with the power-multiplicative norm $\lambda(\alpha)$,
and choose $\pi = \sum_{i=0}^\infty p^i [\overline{\pi_i}]$
with $\alpha(\overline{\pi_0}) \leq p^{-1}$ and 
$\overline{\pi_1}$ a unit in $R$.
Then the map
\[
H(\cdot, \pi, \cdot): 
\calM(R) \times [0,1] 
\to \calM(W(R))
\]
defined by
Theorem~\ref{T:construct Gauss}(a) is continuous.
\end{theorem}
\begin{proof}
To check continuity, we must check that for each 
$x \in W(R)$ and each $\epsilon > 0$,
the sets
\begin{gather*}
\{(\gamma,t) \in \calM(R) \times [0,1]:
H(\gamma,\pi,t)(x) > \epsilon\}, \\
\{(\gamma,t) \in \calM(R) \times [0,1]:
H(\gamma,\pi,t)(x) < \epsilon\}
\end{gather*}
are open.
Pick $(\gamma_0,t_0)$ in one of these sets.
By Corollary~\ref{C:stable presentation},
we can find $\gamma_0$-stable elements $x_0,\dots,x_j \in W(R)$
such that 
\[
\lambda(\gamma_0) \left(
x -\sum_{i=0}^{j} x_i \pi^i \right) 
< \epsilon/2.
\]
We may further ensure that 
each nonzero $x_i$ satisfies $\lambda(\gamma_0)(x_i) > 0$.

Given some nonzero $x_i$, write $x_i = \sum_{k=0}^\infty p^k [\overline{x_{ik}}]$.
Choose an integer $h$ for which $\gamma_0(\overline{x_{i0}}) > p^{-h}$.
Then the set of $\gamma \in \calM(W(R))$ 
for which $\gamma(\overline{x_{i0}}) > p^{-h}$
and $p^{-k} \gamma(\overline{x_{ik}}) < \gamma(\overline{x_{i0}})$ 
for $k=1,\dots,h-1$
is open and contains $\gamma_0$. 
Consequently, there is an open neighborhood $U$ of $\gamma_0$ in
$\calM(W(R))$ such that $x_0,\dots,x_j$ are $\gamma$-stable
for each  $\gamma \in U$. 

For $(\gamma,t) \in U \times [0,1]$,
applying Theorem~\ref{T:construct Gauss}(b) over the ring $\gotho_{\calH(\gamma)}$ yields
\[
 \max\{ \epsilon/2, 
H(\gamma,\pi, t)(x)\}
= \max\{\epsilon/2, 
\max_i \{(t/p)^i \gamma(\overline{x_{i0}})\}\}.
\]
There thus exist an open neighborhood $V$
of $\gamma_0$ and an open interval $I$ containing $t_0$ 
for which for each pair $(\gamma,t) \in V \times I$,
$H(\gamma,\pi, t)(x)$ and $H(\gamma_0,\pi, t_0)(x)$ 
are either both greater than $\epsilon$
or both less than $\epsilon$.
This yields the desired result.
\end{proof}
\begin{cor} \label{C:continuous section}
With notation as in Theorem~\ref{T:continuous1},
the map $H(\cdot, \pi,0)$ induces a homeomorphism 
$\calM(R) \to \calM(W(R)/(\pi))$,
whose inverse is induced by $\mu$. 
Moreover, any subset of $\calM(R)$ is Weierstrass (resp.\ Laurent, rational) if and only if
its image in $\calM(W(R)/(\pi))$ is.
\end{cor}
\begin{proof}
The first statement is immediate from Theorem~\ref{T:continuous1},
Theorem~\ref{T:lifting}, and Remark~\ref{R:radius 0}.
For the second statement, observe that from the proof of 
Theorem~\ref{T:lifting}(a), the image of a Weierstrass
(resp.\ Laurent, rational) subspace of $\calM(R)$ is again one.
We establish the converse only for a rational subspace, as the other cases behave similarly;
we may also assume that $R$ is complete under $\alpha$.
Let
\[
U = \{\gamma \in \calM(W(R)/(\pi)): \gamma(f_i) \leq p_i \gamma(g) \quad
(i=1,\dots,n)\}
\]
be the rational subspace defined by some $f_i, g \in W(R)$ generating the unit ideal in $W(R)/(\pi)$
and some $p_i > 0$. Apply Remark~\ref{R:rational} to find $\epsilon > 0$ for which $\gamma(g) > \epsilon$
for all $\gamma \in U$. By Lemma~\ref{L:stable approximation},
we can find $\overline{f_1}',\dots,\overline{f_n}',\overline{g}' \in R$ such that for all $\gamma \in \calM(W(R)/(\pi))$,
\[
\gamma(f_i - [\overline{f_i}']) \leq p^{-1} \max\{\gamma(f_i), p_i \epsilon \},
\qquad
\gamma(g - [\overline{g}']) \leq p^{-1} \max\{\gamma(g), \epsilon\}.
\]
(More precisely, apply Lemma~\ref{L:stable approximation} with $x = f_1,\dots,f_n,g$, let
$f_1',\dots,f_n',g'$ be the resulting values of $y$, then reduce modulo $p$.)
By Remark~\ref{R:rational}, $[\overline{f_1}'], \dots, [\overline{f_n}'], [\overline{g}']$ also generate the unit ideal
in $W(R)/(\pi)$, so $\overline{f_1}',\dots,\overline{f_n}',\overline{g}',\overline{\pi}$ generate the unit ideal in $R$;
the same is then true without $\overline{\pi}$.

For $\gamma \in \calM(W(R)/(\pi))$ corresponding to $\beta \in \calM(R)$,
$\gamma(g) \geq \epsilon$ if and only if $\beta(\overline{g}') \geq \epsilon$,
in which case $\gamma(g) = \beta(\overline{g}')$. Also, in this case,
$\gamma(f_i) \leq p_i \gamma(g)$ if and only if
$\beta(\overline{f_i}') \leq p_i \gamma(g) = p_i \beta(\overline{g}')$.
Consequently, $U$ corresponds to the rational subspace
\[
\{ \beta \in \calM(R): \beta(\overline{f_i}') \leq p_i \beta(\overline{g}') \quad (i=1,\dots,n)\},
\]
as desired.
\end{proof}

\begin{remark}
Corollary~\ref{C:continuous section} defines a remarkable
section of the projection $\mu$: it is a homeomorphism of topological
spaces, 
but one of the underlying rings is of characteristic $p$
while the other is not. We plan to explore the relationship between these
rings in subsequent work.
\end{remark}

To use Theorem~\ref{T:continuous1} to define the desired homotopy,
we argue as in Remark~\ref{R:Gauss approach}. However, we must overcome
a technical complication that does not occur there, because the 
analogous construction here is not \emph{a priori} well-defined.

\begin{lemma} \label{L:no choice}
Define $\psi: W(R)[T] \to W(R[\overline{T}]^{\perf})$ as in
Lemma~\ref{L:Witt split}. Choose $\beta_1, \beta_2 \in
\calM(W(R[\overline{T}]^{\perf}))$ with $\beta_1(p-[\overline{T}]) = 
\beta_2(p-[\overline{T}]) = 0$ and $\psi^*(\beta_1) = \psi^*(\beta_2)$.
Then for all $t \in [0,1]$, 
the restrictions of $H(\mu(\beta_1), [\overline{T}], t)$
and $H(\mu(\beta_2), [\overline{T}], t)$ to $W(R)$ coincide.
\end{lemma}
\begin{proof}
By Lemma~\ref{L:surjective2},
for $S = W(R[\overline{T}]^{\perf}) \otimes_{W(R)}
W(R[\overline{T}]^{\perf})$,
there exists $\beta_3 \in 
\calM(S)$ restricting to $\beta_1, \beta_2$ on the tensorands.
(For $\beta = \psi^*(\beta_1)$,
one can also argue directly that $\calH(\beta_1) \widehat{\otimes}_{\calH(\beta)}
\calH(\beta_2) \neq 0$ using the fact that $\calH(\beta_i)$ is the completion of
an algebraic extension of $\calH(\beta)$.)
We may identify $S$ with a dense subring of 
$W(R[\overline{T_1}, \overline{T_2}]^{\perf})$ by identifying $[\overline{T}] \otimes 1$
with $[\overline{T_1}]$ and $1 \otimes [\overline{T}]$ with $[\overline{T_2}]$;
we may then extend $\beta_3$ to $W(R[\overline{T_1}, 
\overline{T_2}]^{\perf})$ by continuity.

For $i = 1,2,3$, put $\alpha_i = \mu(\beta_i)$, 
let $\gotho_i$ be the valuation ring of $\calH(\alpha_i)$,
and extend $\beta_i$ to a multiplicative seminorm on $W(\gotho_i)$.
Then $\beta_3(p-[\overline{T_1}]) = \beta_3(p - [\overline{T_2}]) = 0$,
so by Remark~\ref{R:radius 0b}, 
we have $H(\alpha_3, \overline{T_1}, t) = H(\alpha_3, \overline{T_2},t)$
for all $t \in [0,1]$.
Since $H(\alpha_3,\overline{T_i},t)$ restricts to
$H(\alpha_i,\overline{T_i}, t)$, this proves the claim.
\end{proof}

\begin{defn} \label{D:deformation}
Define $\psi: W(R)[T] \to W(R[\overline{T}]^{\perf})$ as in
Lemma~\ref{L:Witt split}.
Given $\beta \in \calM(W(R))$,
restrict $\beta$ along $W(R)[T] \to W(R)[T]/(p-T) \cong W(R)$,
then apply Lemma~\ref{L:Witt split}(b) to extend $\beta$
to $\beta_1 \in \calM(W(R[\overline{T}]^{\perf}))$.
By Lemma~\ref{L:no choice}, for $t \in [0,1]$, the restriction of 
$H(\mu(\beta_1),\overline{T},t)$ to $W(R)$ is independent
of the choice of $\beta_1$; we call this restriction
$H(\beta,t)$. It is a multiplicative seminorm by Theorem~\ref{T:construct Gauss}(a); its formation is evidently compatible with restriction along bounded
homomorphisms.
\end{defn}

\begin{remark} \label{R:use all extensions}
For $\beta \in \calM(W(R))$,
let $\tilde{\beta}$ be
the spectral seminorm associated to the
product seminorm on $W(R)[T]/(T-p) \otimes_{W(R)[T]} 
W(R[\overline{T}]^{\perf})$ using $\beta$ on the first factor;
this equals the supremum 
over all extensions of $\beta$ to $W(R[\overline{T}]^{\perf})$
(see Definition~\ref{D:transform}).
Consequently, by
Lemma~\ref{L:no choice}, we may compute $H(\beta,t)$
by restricting the spectral seminorm associated to the quotient norm on
\[
W(R[\overline{T}]^{\perf})[U]/(U - p + [\overline{T}])
\]
induced by the $(t/p)$-Gauss extension of $\lambda(\mu(\tilde{\beta}))$.
\end{remark}

\begin{remark} \label{R:monotonicity}
One consequence of Remark~\ref{R:use all extensions}
is
monotonicity:
for $\beta, \beta' \in \calM(W(R))$ and
$t, t' \in [0,1]$ with $\beta \geq \beta'$ and 
$t \geq t'$, we have
$H(\beta, t) \geq H(\beta', t')$.
This is not evident from Definition~\ref{D:deformation} because
Lemma~\ref{L:Witt split} does not guarantee 
that $\beta, \beta'$ admit extensions
$\beta_1, \beta'_1$ to $W(R[\overline{T}]^{\perf})$
which satisfy $\beta_1 \geq \beta'_1$.
\end{remark}

We obtain the following analogue of Theorem~\ref{T:homotopy1}.
\begin{theorem} \label{T:final homotopy}
The map $H: \calM(W(R)) \times [0,1] \to \calM(W(R))$ given in
Definition~\ref{D:deformation} is continuous and 
has the following additional properties.
\begin{enumerate}
 \item[(a)]
For $\beta \in \calM(W(R))$, $H(\beta, 0) = \beta$.
\item[(b)]
For $\beta \in \calM(W(R))$, $H(\beta, 1) = (\lambda \circ \mu)(\beta)$.
\item[(c)]
For $\beta \in \calM(W(R))$ and $t \in [0,1]$, $\mu(H(\beta, t)) = \mu(\beta)$.
\item[(d)]
For $\beta \in \calM(W(R))$ and $s,t \in [0,1]$,
$H(H(\beta,s),t) = H(\beta, \max\{s,t\})$.
\end{enumerate}
\end{theorem}
\begin{proof}
Let $\alpha$ be the $\overline{T}$-adic norm on
$R[\overline{T}]^{\perf}$ 
for the normalization $\alpha(\overline{T}) = p^{-1}$.
Equip $W(R[\overline{T}]^{\perf})/(p - [\overline{T}])$
with the quotient norm induced by $\lambda(\alpha)$. We then obtain
a continuous map
\[
\calM(W(R[\overline{T}]^{\perf})/(p - [\overline{T}]))
\times [0,1] \to \calM(W(R))
\]
by applying $\mu \times \id$ (which is continuous by Theorem~\ref{T:lifting}),
then $H(\cdot,\overline{T},\cdot)$ (which is continuous by
Theorem~\ref{T:continuous1}), then restricting along
the inclusion $W(R) \to W(R[\overline{T}]^{\perf})$.

By Lemma~\ref{L:no choice}, we have a commutative diagram
\[
\xymatrix{
\calM(W(R[\overline{T}]^{\perf})/(p - [\overline{T}])) \times [0,1] 
\ar[rd] \ar[d] & \\
\calM(W(R)) \times [0,1] \ar^{H}[r] & \calM(W(R))
}
\]
in which the diagonal arrow is continuous and
the vertical arrow is a quotient map by Lemma~\ref{L:Witt split}(b).
This yields the continuity of $H$.
We deduce (a) from Remark~\ref{R:radius 0},
(b) from Theorem~\ref{T:construct Gauss}(d),
and (c) from Remark~\ref{R:why stable}
(or more precisely, by noting that $\mu(H(\beta, t))(\overline{x}) = H(\beta,t)([\overline{x}])$
and that because Teichm\"uller elements are stable,
the latter equals $\beta([\overline{x}]) = \mu(\beta)(\overline{x})$).

To establish (d), we may follow the construction of 
Definition~\ref{D:deformation} to reduce to the case where
$R = \gotho$ is the valuation ring of a perfect field complete
for a multiplicative norm $\gamma$, and 
$\beta(p-[u])=0$ for some $u \in \gotho$ with $\gamma(u) \leq p^{-1}$.
By Remark~\ref{R:radius 0} again, this ensures that
$\beta = H(\gamma,u,0)$. This formula defines an extension of $\beta$
to $W(\gotho_1)$ whenever $\gotho_1$ is the valuation ring of a complete
field extension of $\Frac \gotho$; we may thus reduce to the case where
$\Frac \gotho$ is algebraically closed.

In this case, by Theorem~\ref{T:factor}, any
nonzero element of $W(\gotho)$ factors as a stable 
element times a product of finitely many terms each of the form 
$p - [u']$ for some $u' \in \gotho$ with $\gamma(u') \leq p^{-1}$.
To establish (d), we thus need only check that the functions
\begin{align*}
f(r) &= -\log H(H(\beta,s),e^{-r})(p-[u']) \\
g(r) &= -\log H(\beta,\max\{s,e^{-r}\})(p-[u'])
\end{align*}
are identically equal.
By Lemma~\ref{L:compute linear}, $f$ and $g$ are both
continuous, concave, nondecreasing, and piecewise linear
with slopes in $\{0,1\}$. They moreover take the same value
at $r=0$ (namely $\log p$) and have the same limiting value
as $r \to \infty$ (because $H(H(\beta,s),0) = H(\beta,s)$
by (a)). Consequently, they must coincide.
\end{proof}
\begin{cor}
Each subset of $\calM(R)$ has the same homotopy type
as its inverse image in $\calM(W(R))$ under $\mu$.
\end{cor}

We have the following analogue of Lemma~\ref{L:homotopy2}.
\begin{lemma} \label{L:same homotopy}
For $\alpha \in \calM(R)$ and $s,t \in [0,1]$,
$H(H(\alpha,u,s),t) = H(\alpha,u,\max\{s,t\})$.
\end{lemma}
\begin{proof}
Put $\beta = H(\alpha,u,0)$ and set notation as in
Definition~\ref{D:deformation}. Then
$\beta_1(p - [u]) = \beta_1(p - [\overline{T}]) = 0$,
so $H(\alpha,u,s) = H(\beta,s)$ by Remark~\ref{R:radius 0b}.
By Theorem~\ref{T:final homotopy}(d),
\[
H(H(\alpha,u,s),t) = H(H(\beta,s),t) = H(\beta, \max\{s,t\})
= H(\alpha,u,\max\{s,t\})
\]
as desired.
\end{proof}

We also have the following analogue of Theorem~\ref{T:dominate}.
Again, this depends on an analysis of the fibres of $\mu$, which we
carry out in \S\ref{sec:fibres}.
\begin{defn}
For $\beta \in \calM(W(R))$, 
the set of $s \in [0,1]$ for which $H(\beta,s) = \beta$ is nonempty
(because it contains $0$), and closed (by continuity), so it has a greatest
element. As in Definition~\ref{D:radius}, we call this greatest element
the \emph{radius} of $\beta$, and denote it by $r(\beta)$.
\end{defn}

\begin{theorem} \label{T:tree}
Suppose $\beta, \gamma \in \calM(W(R))$ are such that $\beta \geq \gamma$
and $\mu(\beta) = \mu(\gamma)$. Then $\beta = H(\gamma, r(\beta))$.
\end{theorem}
\begin{proof}
Put $\alpha = \mu(\beta) = \mu(\gamma)$, let $\gotho$
be the valuation ring of $\calH(\alpha)$,
and identify $\beta, \gamma$ with the corresponding points in
$\mu^{-1}(\alpha) \subseteq \calM(W(\gotho))$. These identifications
are compatible with the formation of $H(\cdot,t)$; in particular,
they do not change the radius of $\beta$. It thus suffices to check
the case $R = \gotho$, for which see Lemma~\ref{L:tree proof}.
\end{proof}
\begin{cor} \label{C:tree}
For $\beta,\gamma \in \calM(W(R))$ satisfying $\mu(\beta) = \mu(\gamma)$
and $\beta \geq \gamma$,
we have  $r(\beta) \geq r(\gamma)$, with equality if and only if
$\beta = \gamma$.
\end{cor}
\begin{proof}
For $t \in [0, r(\gamma)]$, by Theorem~\ref{T:tree}
and Theorem~\ref{T:final homotopy}(d) we have
\[
H(\beta,t) = H(H(\gamma,r(\beta)),t) = H(H(\gamma,t), r(\beta)) =
H(\gamma,r(\beta)) = \beta,
\]
so $r(\beta) \geq r(\gamma)$. If equality holds, then
$\gamma = H(\gamma,r(\gamma)) = H(\gamma,r(\beta)) = \beta$.
\end{proof}

\section{Structure of fibres}
\label{sec:fibres}

We conclude with a description
of the fibres of the map $\mu: \calM(W(R)) \to \calM(R)$
similar to the description of
$\calM(K[T])$ given in \S\ref{sec:polynomial}. 
This will allow us to establish Theorem~\ref{T:tree}, 
thus giving a combinatorial interpretation of the fibres of $\mu$.

\begin{hypothesis}
Throughout \S\ref{sec:fibres}, retain Hypothesis~\ref{H:Newton polygons}.
In addition, let $\tilde{\gotho}$ be the valuation ring of
the completion of an algebraic closure of
$\Frac \gotho$, equipped with the unique multiplicative extension 
$\tilde{\alpha}$ of $\alpha$,
and equip $W(\tilde{\gotho})$ with the multiplicative
norm $\lambda(\tilde{\alpha})$.
\end{hypothesis}

\begin{defn}
For $u \in \tilde{\gotho}$ 
with $\tilde{\alpha}(u) \leq p^{-1}$ and $t \in [0,1]$, let 
$\tilde{\beta}_{u,t}  \in \mu^{-1}(\tilde{\alpha})$ be the
seminorm $H(\tilde{\alpha},u,t)$
of Theorem~\ref{T:construct Gauss}.
Let $\beta_{u,t}$ be the restriction of $\tilde{\beta}_{u,t}$ to $W(\gotho)$.
\end{defn}

Before studying the $\beta_{u,t}$, we must work out some facts about the
$\tilde{\beta}_{u,t}$ which are not quite as obvious as their counterparts
for $K[T]$.
\begin{lemma} \label{L:sup norm extended}
For $u,u' \in \tilde{\gotho}$ 
with $\tilde{\alpha}(u), \tilde{\alpha}(u') \leq p^{-1}$ and $t \in (0,1]$,
the following conditions are equivalent.
\begin{enumerate}
\item[(a)] We have $\tilde{\beta}_{u,t} = \tilde{\beta}_{u',t}$.
\item[(b)] We have $\tilde{\beta}_{u,t} \geq \tilde{\beta}_{u',t}$.
\item[(c)] We have $\tilde{\beta}_{u,t} \geq \tilde{\beta}_{u',0}$.
\item[(d)] We have $t/p \geq \tilde{\beta}_{u',0}(p-[u])$.
\end{enumerate}
\end{lemma}
\begin{proof}
Clearly (a)$\implies$(b)$\implies$(c)$\implies$(d);
it remains to check that (d)$\implies$(a). 
If $t \geq \max\{p \tilde{\alpha}(u), p \tilde{\alpha}(u')\}$, 
then $\tilde{\beta}_{u,t} = \tilde{\beta}_{u',t}$ by
Theorem~\ref{T:construct Gauss}(d),
so (a) always holds. We may thus assume $t < \max\{p \tilde{\alpha}(u),
p \tilde{\alpha}(u')\}$ hereafter.

By (d), we have $\tilde{\beta}_{u',0}([u] - [u']) \leq t/p$.
That is, there exists $y \in W(\tilde{\gotho})$
for which 
\[
\lambda(\tilde{\alpha})([u] - [u'] + y(p-[u'])) \leq t/p.
\]
Note that we cannot have $\tilde{\alpha}(u) \neq \tilde{\alpha}(u')$,
as then $[u]-[u']$ would be stable and we would derive the 
contradiction $\max\{\tilde{\alpha}(u), \tilde{\alpha}(u')\} 
= \lambda(\tilde{\alpha})([u] - [u']) = \tilde{\beta}_{u',0}([u]-[u']) \leq t/p$. We must thus have
$\tilde{\alpha}(u) = \tilde{\alpha}(u')$.
For $\overline{y}$ the reduction of $y$ modulo $p$,
we cannot have $\tilde{\alpha}(1 + \overline{y}) < 1$, or else we would
derive the contradiction
$\max\{\tilde{\alpha}(u), \tilde{\alpha}(u')\} = \tilde{\alpha}(u - (1+\overline{y}) u')
\leq t/p$.
We deduce that $1+y$ is a unit in $W(\tilde{\gotho})$.

Put $y' = y/(1+y)$; then 
\[
[u] - [u'] + y'(p-[u]) = (1+y)^{-1} ([u] - [u'] + y(p - [u'])),
\]
so $\lambda(\tilde{\alpha})([u] - [u'] + y'(p-[u])) \leq t/p$
and hence $\tilde{\beta}_{u,0}(p - [u']) \leq t/p$.
In other words, condition (d) is symmetric in $u$ and $u'$.

This means that to prove that (d)$\implies$(a), it is sufficient to check
that (d)$\implies$(b). Given (d), for $x \in W(\gotho)$,
apply Lemma~\ref{L:stable presentation} to
construct a stable presentation $x_0,x_1,\dots$ of $x$ with respect
to $u$. By Theorem~\ref{T:construct Gauss}(b),
$\tilde{\beta}_{u,t}(x) = \max_i \{(t/p)^i \lambda(\tilde{\alpha})(x_i)\}$.
Applying $\tilde{\beta}_{u',0}$ to the identity
$x = \sum_i x_i (p-[u])^i$ then gives $\tilde{\beta}_{u,t}(x) \geq
\tilde{\beta}_{u',0}(x)$. Lemma~\ref{L:same homotopy} 
and Remark~\ref{R:monotonicity}
then give
\[
\tilde{\beta}_{u,t} = H(\tilde{\beta}_{u,t},t) \geq 
H(\tilde{\beta}_{u',0},t) = \tilde{\beta}_{u',t},
\]
yielding (b) and completing the proof.
\end{proof}
Lemma~\ref{L:sup norm extended} allows us to replace the center
$u$ of the norm $\tilde{\beta}_{u,t}$ with a nearby value,
as was critical in the analysis of $\calM(K[T])$.
\begin{cor} \label{C:nearby}
For $u,u' \in \tilde{\gotho}$ 
with $\tilde{\alpha}(u), \tilde{\alpha}(u') \leq p^{-1}$ and $t \in (0,1]$,
if $\lambda(\alpha)([u] - [u']) \leq t/p$,
then $\tilde{\beta}_{u,t} = \tilde{\beta}_{u',t}$.
\end{cor}
\begin{proof}
Since $\tilde{\beta}_{u,t} \leq \lambda(\alpha)$,
this follows from Lemma~\ref{L:sup norm extended}.
\end{proof}
\begin{cor} \label{C:nearby2}
For $u \in \tilde{\gotho}$ 
with $\tilde{\alpha}(u) \leq p^{-1}$ and $t \in (0,1]$,
there exists $u' \in \tilde{\gotho}$ which is integral over $\gotho$
such that $\tilde{\alpha}(u') \leq p^{-1}$,
$\lambda(\tilde{\alpha})([u] - [u']) < t/p$,
and $\tilde{\beta}_{u,t} = \tilde{\beta}_{u',t}$.
\end{cor}
\begin{proof}
By Remark~\ref{R:addition formula},
$[u] - [u'] = \sum_{i=0}^\infty p^i [P_i]$ for some
polynomials $P_i$ in $u^{p^{-i}}, (u')^{p^{-i}}$ such that
$P_i$ is homogeneous of degree $p^i$ and divisible by $u^{p^{-i}} - (u')^{p^{-i}}$.
It follows that 
\begin{equation} \label{eq:ultrametric}
\lambda(\tilde{\alpha})([u] - [u'])
\leq \max_i \{p^{-i} \tilde{\alpha}(u - u')^{p^{-i}}\}.
\end{equation}
We can make the right side smaller than $t/p$
by ensuring that $\tilde{\alpha}(u - u') < (tp^{i-1})^{p^i}$ for each of the finitely
many nonnegative integers $i$ for which $p^{-i} \geq t/p$;
this is possible because the integral closure of $\gotho$
in $\tilde{\gotho}$ is dense. By Corollary~\ref{C:nearby},
we obtain the desired result.
\end{proof}

\begin{remark} \label{R:ultrametric}
Define the function $d(u,u') = p\tilde{\beta}_{u',0}(p-[u])$.
If $d(u,u'), d(u',u'') \leq t$, then
Lemma~\ref{L:sup norm extended} gives $\tilde{\beta}_{u',t} = \tilde{\beta}_{u,t}
= \tilde{\beta}_{u'',t}$ and hence $d(u,u'') \leq t$.
In other words, the function $d$
satisfies the strong triangle inequality
$d(u,u'') \leq \max\{d(u,u'),d(u',u'')\}$.
Lemma~\ref{L:sup norm extended} also implies the symmetry property
$d(u,u') = d(u',u)$. 
This almost implies that $d$ is an ultrametric distance
function, but not quite:  
we can have $d(u,u') = 0$ even when $u \neq u'$. 
(That is, $d$ is a \emph{pseudometric} rather than a true metric.)
Nonetheless, the function $d$
will play a role in the following arguments similar to that played by
the usual distance function
on $K$ in the analysis of $\calM(K[T])$. 
\end{remark}

We can now give an analogue of Lemma~\ref{L:sup norm same center}.
\begin{lemma} \label{L:sup norm same center2}
For $u \in \tilde{\gotho}$ 
with $\tilde{\alpha}(u) \leq p^{-1}$ and $s,t \in [0,1]$,
$\beta_{u,s} \geq \beta_{u,t}$ if and only if $s \geq t$.
\end{lemma}
\begin{proof}
If $s \geq t$, then evidently $\beta_{u,s} \geq \beta_{u,t}$. It remains
to show that if $s>t$, then $\beta_{u,s} \neq \beta_{u,t}$; it is enough
to check this when $t > 0$. By Corollary~\ref{C:nearby2}, we can
choose $u' \in \tilde{\gotho}$ integral over $\gotho$ 
with $\tilde{\alpha}(u') \leq p^{-1}$ for which
$\beta_{u,t} = \beta_{u',t}$, and hence $\beta_{u,s} = \beta_{u',s}$
by Lemma~\ref{L:same homotopy}.
Let $P(T) = \prod_{i=1}^m (T - u_i)$ 
be the minimal polynomial of $u'$ over $\gotho$. 
Then $\tilde{\beta}_{u',s}(p - [u_i]) \geq \tilde{\beta}_{u',t}(p - [u_i])$
with strict inequality when $u_i = u'$.
If we put $y = \prod_{i=1}^m (p - [u_i]) \in W(\gotho)$, then
$\beta_{u,s}(y) = \tilde{\beta}_{u',s}(y) > \tilde{\beta}_{u',t}(y)
= \beta_{u,t}(y)$, so $\beta_{u,s} \neq \beta_{u,t}$ as desired.
\end{proof}
\begin{cor} \label{C:radius2}
For $u \in \tilde{\gotho}$ 
with $\tilde{\alpha}(u) \leq p^{-1}$ and $t \in [0,1]$,
$r(\beta_{u,t}) = t$.
\end{cor}
\begin{proof}
This follows from Lemma~\ref{L:sup norm same center2}
plus Lemma~\ref{L:same homotopy}.
\end{proof}

We also have an analogue of Lemma~\ref{L:sup norm same radius}.
\begin{lemma} \label{L:sup norm same radius2}
For $u,u' \in \tilde{\gotho}$ 
with $\tilde{\alpha}(u), \tilde{\alpha}(u') \leq p^{-1}$ and $t \in [0,1]$, 
the following are equivalent.
\begin{enumerate}
\item[(a)] We have $\beta_{u,t} = \beta_{u',t}$.
\item[(b)] We have $\beta_{u,t} \geq \beta_{u',t}$.
\item[(c)] We have $\beta_{u,t} \geq \beta_{u',0}$.
\item[(d)] There exists $\tau \in \Aut(\tilde{\gotho}/\gotho)$
for which $t/p \geq \tilde{\beta}_{u',0}(p-[\tau(u)])$.
\end{enumerate}
\end{lemma}
\begin{proof}
Assume first that $t>0$.
By Lemma~\ref{L:sup norm extended}, we have
(d)$\implies$(a)$\implies$(b)$\implies$(c), so it remains to check that
(c)$\implies$(d). Assume (c), then
apply Corollary~\ref{C:nearby2} to construct $v \in \tilde{\gotho}$
integral over $\gotho$ with $\tilde{\alpha}^{-1}(v) = p^{-1}$
for which $\lambda(\tilde{\alpha})([u] - [v]) < t/p$ and 
$\tilde{\beta}_{u,t} = \tilde{\beta}_{v,t}$.
Let $P(T) = \prod_{i=1}^m (T - v_i)$ 
be the minimal polynomial of $v$ over $\gotho$,
with the roots ordered so that the sequence 
$t_i = p \tilde{\beta}_{u',0}(p - [v_i])$
is nondecreasing.

If (d) fails, then also $t/p < \tilde{\beta}_{u',0}(p-[\tau(v)])$,
so $t_i > t$ for $i=1,\dots,m$.
We exploit transitivity as in Remark~\ref{R:ultrametric}:
since $t_i \geq t_1$, by Lemma~\ref{L:sup norm extended} we have
$\tilde{\beta}_{u',t_i} = \tilde{\beta}_{v_i,t_i}$
and
$\tilde{\beta}_{u',t_i} = \tilde{\beta}_{v_1,t_i}$,
so $\tilde{\beta}_{v_i,t_i} = \tilde{\beta}_{v_1,t_i}$.
By Lemma~\ref{L:sup norm extended} again,
\begin{equation} \label{eq:bound linear}
\max\{t/p, \tilde{\beta}_{v_1,0}(p - [v_i])\} \leq t_i/p.
\end{equation}
This inequality becomes strict for $i=1$.

If we put $y = \prod_{i=1}^m (p - [v_i]) \in W(\gotho)$, then
by Lemma~\ref{L:compute linear} and \eqref{eq:bound linear},
\begin{align*}
\beta_{u,t}(y) = \tilde{\beta}_{v_1,t}(y) 
&= \prod_{i=1}^m \tilde{\beta}_{v_1,t}(p - [v_i]) \\
&= \prod_{i=1}^m \max\{t/p, \tilde{\beta}_{v_1,0}(p - [v_i])\} \\
&< \prod_{i=1}^m (t_i/p) = \prod_{i=1}^m \tilde{\beta}_{u',0}(p - [v_i]) = \beta_{u',0}(y),
\end{align*}
contradiction. Hence (d) holds, as desired.

Suppose now that $t=0$. 
Note that each condition for $t=0$ implies the corresponding condition for all $t>0$.
For (a),(b),(c), the converse implication is clear; the converse implication also holds for (d)
by the completeness of $\tilde{\gotho}$
and the compactness of $\Aut(\tilde{\gotho}/\gotho)$. We may thus reduce the claim to the case $t>0$
treated above. 
\end{proof}

We are now ready to make the decisive step, analogous to
Lemma~\ref{L:dominate}.
\begin{lemma} \label{L:dominate2}
For $\beta \in \mu^{-1}(\alpha)$ and $s \in (r(\beta),1]$,
there exists $u \in \gotho$ with $\tilde{\alpha}(u) \leq p^{-1}$
for which $H(\beta,s) = \beta_{u,s}$.
\end{lemma}
\begin{proof}
Let $S$ be the set of $s \in [0,1]$ for which $\beta_{u,s} \geq \beta$
for some $u \in \tilde{\gotho}$ with $\tilde{\alpha}(u) \leq p^{-1}$. 
The set $S$ is up-closed and nonempty;
let $t$ be its infimum.
As in the proof of Lemma~\ref{L:dominate}, 
it suffices to check that $r(\beta) \geq t$.

By proceeding as in Definition~\ref{D:deformation}, we can construct
an isometric embedding $\gotho \to \gotho_1$ of complete valuation rings
of characteristic $p$, with the norm on $\gotho_1$ denoted by $\alpha_1$,
and an element $v \in \gotho_1$ with $\alpha_1(v) = p^{-1}$,
for which $\beta$ is the restriction of the seminorm $H(\alpha_1,v,0)$.
There is no harm in further enlarging $\gotho_1$ so that
$\Frac \gotho_1$ becomes algebraically closed; we may then  
identify $\tilde{\gotho}$ with a subring of $\gotho_1$.

For $u \in \gotho$, if $\tilde{\alpha}(u) < p^{-1}$, then
by Lemma~\ref{L:compute linear}, $H(\alpha_1,v,s)(p - [u])$
is constant on $[0,1]$. If instead $\tilde{\alpha}(u) = p^{-1}$,
then for $s \in [0,t)$ we have $H(\beta,s) \neq \beta_{u,s}$,
so by Lemma~\ref{L:sup norm same radius2}, $s/p < H(\alpha_1,v,0)(p - [u])$.
By Lemma~\ref{L:compute linear}, for $s \in [0,t]$,
\[
H(\alpha_1,v,s)(p-[u]) = \max\{s/p, H(\alpha_1,v,0)(p - [u])\}
= H(\alpha_1,v,0)(p - [u])\}.
\]
For each nonzero $x \in W(\gotho)$,
by Theorem~\ref{T:factor} we have
$x = y(p - [u_1]) \cdots (p - [u_n])$
for some stable $y \in W(\tilde{\gotho})$
and some $u_1,\dots,u_n \in \tilde{\gotho}$
with $\tilde{\alpha}(u_i) \leq p^{-1}$.
For $s \in [0,t]$,
\[
H(\beta,s)(x) = 
H(\alpha_1,v,s)(x) = \lambda(\tilde{\alpha})(y) \prod_{i=1}^n H(\alpha_1,v,0)(p - [u_i])
\]
is independent of $s$. Hence $H(\beta,s) = \beta$ for $s \in [0,t]$,
and so $r(\beta) \geq t$ as desired.
\end{proof}
\begin{cor} \label{C:infimum2}
Suppose that $\beta \in \mu^{-1}(\alpha)$ is such that
$\beta \neq \beta_{u,t}$ for all $u \in \tilde{\gotho}$ 
with $\tilde{\alpha}(u) \leq p^{-1}$ and all $t \in [0,1]$.
Then for each $y \in W(\gotho)$, for any sufficiently small 
$s \in (r(\beta),1]$,
$\beta(y) = H(\beta,s)(y)$.
\end{cor}

With this analysis, we obtain Theorem~\ref{T:tree} as follows.
\begin{lemma}  \label{L:tree proof}
Theorem~\ref{T:tree} holds in case $R = \gotho$.
\end{lemma}
\begin{proof}
If $r(\beta) = 1$, then $\beta = H(\beta,1) = H(\gamma,1)$
by Theorem~\ref{T:final homotopy}(b).
If $r(\gamma) = 1$, then by Theorem~\ref{T:final homotopy}(b) again, 
$\beta \geq \gamma = H(\gamma,1) = H(\beta,1) \geq \beta$
and so $\beta = H(\gamma,1)$.
It is thus safe to assume $r(\beta), r(\gamma) < 1$.

For each $s \in (\max\{r(\beta),r(\gamma)\},1]$, 
by Lemma~\ref{L:dominate2} we
have $H(\beta,s) = \beta_{u,s}$,
$H(\gamma,s) = \beta_{u',s}$ for some $u,u' \in \tilde{\gotho}$
with $\tilde{\alpha}(u),\tilde{\alpha}(u') \leq p^{-1}$.
Since $\beta \geq \gamma$ implies
$H(\beta,s) \geq H(\gamma,s)$ by Remark~\ref{R:monotonicity}, we have $\beta_{u,s} \geq \beta_{u',s}$,
but by Lemma~\ref{L:sup norm same radius2},
this forces $\beta_{u,s} = \beta_{u',s}$.
Hence $H(\beta,s) = H(\gamma,s)$.

If $r(\gamma) > r(\beta)$, by taking the limit as $s \to r(\gamma)^+$,
we deduce that $\gamma = H(\beta,r(\gamma)) = H(\beta,r(\beta)) = \beta$,
contradiction. Hence $r(\beta) \geq r(\gamma)$, and by taking the limit as $s \to r(\beta)^+$,
we deduce $\beta = H(\gamma,r(\beta))$ as desired.
\end{proof}

We derive the following corollary analogous to 
Corollary~\ref{C:compatible extension}.
\begin{cor}
For any $\beta, \gamma \in \mu^{-1}(\alpha)$ with $\beta \geq \gamma$,
there exist $\tilde{\beta}, \tilde{\gamma} \in \mu^{-1}(\tilde{\alpha})$
restricting to
$\beta, \gamma$, respectively, for which $\tilde{\beta} \geq \tilde{\gamma}$.
\end{cor}
\begin{proof}
Extend $\gamma$ as in the proof of Theorem~\ref{T:final homotopy}(d),
then put $\tilde{\beta} = H(\tilde{\gamma}, r(\beta))$; this restricts
to $\beta$ by Theorem~\ref{T:tree}.
\end{proof}

To obtain an analogue of Corollary~\ref{C:disc}, we must make the 
function $d(u,u')$ from Remark~\ref{R:ultrametric} more explicit.
\begin{lemma} \label{L:compute metric}
Consider $u,u' \in \tilde{\gotho}$ with 
$\tilde{\alpha}(u), \tilde{\alpha}(u') \leq p^{-1}$.
\begin{enumerate}
\item[(a)]
If $u=0$ or $\tilde{\alpha}(u'-u) > p^{-p/(p-1)} \max\{\tilde{\alpha}(u), \tilde{\alpha}(u')\}$, then
$d(u,u') = p\tilde{\alpha}(u'-u)$.
\item[(b)]
If $u \neq 0$ and there exists a nonnegative integer $i$ for which
$\tilde{\alpha}(u'/u-1) \in (p^{-p^{i+1}/(p-1)}, p^{-p^i/(p-1)})$, then
\[
d(u,u') = p^{1-i} \tilde{\alpha}(u)
\tilde{\alpha}(u'/u-1)^{p^{-i}} \in (p^{-i-1/(p-1)} \tilde{\alpha}(u), p^{-(i-1)-1/(p-1)} \tilde{\alpha}(u)).
\]
\item[(c)]
If $u \neq 0$ and there exists a positive integer $i$ for which
$\tilde{\alpha}(u'/u-1) = p^{-p^{i}/(p-1)}$, then
$d(u,u') \leq p^{-(i-1)-1/(p-1)} \tilde{\alpha}(u)$, with equality unless 
$\tilde{\alpha}(u) = p^{-1}$ and $\tilde{\alpha}(1 - u^{p^{i}}(u'/u-1)^{1-p}) < 1$.
\end{enumerate}
\end{lemma}
\begin{proof}
Part (a) is clear when $u=0$, so we may assume $u \neq 0$ throughout.
Write
\begin{equation} \label{eq:compute distance1}
[u'/u] - 1 = \sum_{j=0}^\infty p^j [P_j((u'/u - 1)^{p^{-j}})]
\end{equation}
for $P_j(T) \in \Fp[T]$ as in Lemma~\ref{L:Witt explicit}.
By Lemma~\ref{L:Witt explicit}, $P_j(T)$ is divisible by $T$ but not
by $T^2$; consequently, if $\tilde{\alpha}(u'/u - 1) < 1$, 
then 
\[
\tilde{\alpha}(P_j((u'/u-1)^{p^{-j}})) = \tilde{\alpha}(u'/u - 1)^{p^{-j}}
\]
and so
\begin{equation} \label{eq:compute distance}
\lambda(\tilde{\alpha})([u'] - [u]) = \max_j \{p^{-j} \tilde{\alpha}(u)
\tilde{\alpha}(u'/u-1)^{p^{-j}} \}.
\end{equation}
In case (a), the maximum in \eqref{eq:compute distance} is achieved only by the
index $j=0$; in case (b), the maximum is achieved only by $j=i$.
In these cases, the right side of \eqref{eq:compute distance1} is dominated under $\lambda(\tilde{\alpha})$
by a single term which is a power of $p$ times a Teichm\"uller element, so this term also dominates
under $\tilde{\beta}_{u,0}$. This yields the desired results in these cases.

In case (c), the maximum in
\eqref{eq:compute distance} is achieved only by the indices $j=i-1,i$.
We modify the presentation of
$[u'/u]-1$ by replacing $p^{i} [P_{i}((u'/u - 1)^{p^{-i}})]$
with $p^{i-1} [u P_{i}((u'/u - 1)^{p^{-i}})]$. We then observe that
$d(u,u') \leq p^{-(i-1)-1/(p-1)} \tilde{\alpha}(u)$
with equality
unless
\[
\tilde{\alpha}((u'/u-1)^{p^{-i}} - u (u'/u-1)^{p^{-i+1}}) <
p^{-1/(p-1)},
\]
which yields the desired result.
\end{proof}
\begin{cor} \label{C:many dominated}
For $u \in \tilde{\gotho}$ 
with $\tilde{\alpha}(u) \leq p^{-1}$ and $0 < s < t \leq 1$,
there are infinitely many points of $\mu^{-1}(\tilde{\alpha})$ of the form
$\tilde{\beta}_{u',s}$ which are dominated by $\tilde{\beta}_{u,t}$.
\end{cor}
\begin{proof}
Suppose first that $u=0$; then for any $u' \in \tilde{\gotho}$,
we have
$\tilde{\beta}_{u,0}(p - [u']) = \tilde{\alpha}(u')$
because $\tilde{\beta}_{u,0}(p) = 0$.
Choose $s' \in (s,t)$
for which $s'/p$ occurs as the norm of some element of $\tilde{\gotho}$.
As in the proof of Lemma~\ref{L:disc}, 
we can find
an infinite subset $S$ of $\tilde{\gotho}$ such that
$\tilde{\alpha}(u') = \tilde{\alpha}(u'-u'') = s'/p$ for all distinct
$u',u'' \in S$. We then have $\tilde{\beta}_{u,t} \geq \tilde{\beta}_{u',s}$
for all $u' \in S$ by Lemma~\ref{L:sup norm extended}. 
Moreover, for $u', u'' \in S$ distinct, 
$[u'] - [u'']$ is stable, so $\tilde{\beta}_{u',0}(p - [u''])
= \tilde{\beta}_{u',0}([u'] - [u'']) = s'/p$ and hence
$\tilde{\beta}_{u',s} \neq \tilde{\beta}_{u'',s}$
by Lemma~\ref{L:sup norm extended} again.

Suppose next that $u \neq 0$. Choose $c \in (p^{-p/(p-1)}, p^{-1/(p-1)})$ 
occurring as the norm of an element of $\tilde{\gotho}$
and such that $p^{-i+1} \tilde{\alpha}(u) c \in (s,t)$ for some
nonnegative integer $i$.
Again as in Lemma~\ref{L:disc},
we choose an infinite subset
$S$ of $\tilde{\gotho}$ such that
$\tilde{\alpha}(u'/u-1) = \tilde{\alpha}(u'/u-u''/u) = c^{p^i}$ for all distinct
$u',u'' \in S$. By Lemma~\ref{L:compute metric}, we have
$\tilde{\beta}_{u,0}(p - [u']) = \tilde{\beta}_{u',0}(p - [u'']) = 
p^{-i} \tilde{\alpha}(u) c$ for all distinct $u',u'' \in S$.
By Lemma~\ref{L:sup norm extended},
$\tilde{\beta}_{u,t} \geq \tilde{\beta}_{u',s}$
for all $u' \in S$, and 
$\tilde{\beta}_{u',s} \neq \tilde{\beta}_{u'',s}$ for all distinct
$u', u'' \in S$.
\end{proof}

We can now derive an analogue of Lemma~\ref{L:disc}.
\begin{lemma}
For $u \in \tilde{\gotho}$ 
with $\tilde{\alpha}(u) \leq p^{-1}$ and $t \in [0,1]$, 
let $D(u,t)$ be the set of
$\beta_{v,0} \in \mu^{-1}(\alpha)$ dominated by $\beta_{u,t}$.
Then for $s,t \in [0,1]$, $D(u,s) = D(u,t)$ if and only if $s=t$.
\end{lemma}
\begin{proof}
It suffices to deduce a contradiction under the assumption that
$D(u,s) = D(u,t)$ for some $t>s>0$. By Corollary~\ref{C:nearby2},
we can find $u' \in \tilde{\gotho}$ integral over $\gotho$ for which
$\tilde{\beta}_{u,s} = \tilde{\beta}_{u',s}$, so that $D(u,t) = D(u',t)$
and $D(u,s) = D(u',s)$. Since $D(u,t) = D(u,s)$, for any $\beta_{v,0}
\in D(u,t)$, we have $\beta_{u',s} \geq \beta_{v,0}$
and hence (by Lemma~\ref{L:sup norm same radius2})
$\tilde{\beta}_{v,0}(p - [\tau(u')]) \leq s/p$ for some $\tau \in 
\Aut(\tilde{\gotho}/\gotho)$. Consequently, there are only finitely
many points in $\mu^{-1}(\tilde{\alpha})$ of the form $\tilde{\beta}_{v,s}$
which are dominated by $\tilde{\beta}_{u,t}$; however, 
this would contradict Corollary~\ref{C:many dominated}. This contradiction
establishes the desired result.
\end{proof}

We also derive the following analogue of Theorem~\ref{T:classify1}.
\begin{theorem} \label{T:classify2}
Each element of $\mu^{-1}(\alpha)$ is of 
exactly one of the following four types.
\begin{enumerate}
\item[(i)]
A point of the form $\beta_{u,0}$ for some $u \in \tilde{\gotho}$ 
with $\tilde{\alpha}(u) \leq p^{-1}$. Such a point has radius $0$ and is minimal.
\item[(ii)]
A point of the form $\beta_{u,t}$ for some $u \in \tilde{\gotho}$ 
with $\tilde{\alpha}(u) \leq p^{-1}$
and some $t \in (0,1)$ such that $t/p$ is the norm of an element of $\tilde{\gotho}$.
Such a point has radius $t$ and is not minimal.
\item[(iii)]
A point of the form $\beta_{u,t}$ for some $u \in \tilde{\gotho}$ 
with $\tilde{\alpha}(u) \leq p^{-1}$
and some $t \in (0,1)$ such that $t/p$ is not the norm of an element of 
$\tilde{\gotho}$.
Such a point has radius $t$ and is not minimal.
\item[(iv)]
The infimum of a 
sequence $\beta_{u_i,t_i}$ for which the sequence $D(u_i,t_i)$
is decreasing with empty intersection. 
Such a point has radius $\inf_i \{t_i\} > 0$
and is minimal.
\end{enumerate}
\end{theorem}
\begin{proof}
By Corollary~\ref{C:radius2}, $r(\beta_{u,t}) = t$. Consequently,
types (i), (ii), (iii) are mutually exclusive. Moreover, 
$\beta_{u,t}$ cannot be of type (iv), as otherwise
$\beta_{u,0}$ would belong to the empty intersection $\cap_i D(u_i,t_i)$.
Consequently, no point can be of more than one type.

It remains to check that any point $\beta \in \mu^{-1}(\alpha)$
not of the form $\beta_{u,t}$
is of type (iv) and is minimal of the claimed radius.
Choose a sequence $1 \geq t_1 > t_2 > \cdots$ with infimum $r(\beta)$.
By Lemma~\ref{L:dominate2}, for each $i$, we have
$H(\beta,t_i) = \beta_{u_i,t_i}$ for some $u_i \in \tilde{\gotho}$.
The sequence $\beta_{u_1,t_1}, \beta_{u_2,t_2}, \dots$ is decreasing
with infimum $\beta$; the sequence $D(u_i,t_i)$ is also decreasing.
For each $u \in \tilde{\gotho}$, there exists $i$ for which
for which $\beta_{u,t_i} \neq \beta_{u_i, t_i}$;
for such $i$ we have $\beta_{u,0} \notin D(u_i,t_i)$ by
Lemma~\ref{L:sup norm same radius2}.
Hence the $D(u_i,t_i)$ have empty intersection. 
Hence $\beta$ is of type (iv); it is minimal by
Theorem~\ref{T:tree} plus Lemma~\ref{L:dominate2}.
Since $\beta = \inf_i\{\beta_{u_i,t_i}\}$
and $r(\beta_{u_i,t_i}) = t_i$ by Corollary~\ref{C:radius2},
we have $r(\beta) \geq \inf_i \{t_i\}$;
the reverse inequality also holds because $t_i = r(\beta_{u_i,t_i}) \geq r(\beta)$
by Theorem~\ref{T:tree}.

Suppose by way of contradiction that $r(\beta) =0$.
By Corollary~\ref{C:nearby2}, we may choose the $u_i$ to be integral over $\gotho$.
Let $U_0$ denote the original sequence $u_1, u_2,\dots$.
For $h=1,2,\dots$, we construct a subsequence $U_h$ of $U_{h-1}$ such that
any two terms $v_1, v_2$ of $U_h$ satisfy $d(v_1, v_2) \leq p^{-h}$,
as follows. Given $U_{h-1}$, for $i$ sufficiently large, whenever $u_i \in U_{h-1}$,
we have $t_i \leq p^{-h}$ and $\beta_{u_i,t_i} = \beta_{u_j,t_i}$ for all $j \geq i$
with $u_j \in U_{h-1}$. By Lemma~\ref{L:sup norm same center2} and the integrality of $u_i$
over $\gotho$, this limits the 
$u_j$ to finitely many closed discs of radius $t_i$ under $d$. One of these discs
then contains infinitely many elements of $U_{h-1}$; choose these to form the subsequence
$U_h$.

By diagonalizing (i.e., choosing a subsequence of $U_0$ whose $i$-th term belongs to $U_i$
for each $i$), we obtain a Cauchy sequence in $\tilde{\gotho}$ with respect to $d$.
By Lemma~\ref{L:complete pseudometric} below, this sequence admits a limit $u$ with respect to $d$,
which then satisfies $\beta = \beta_{u,0}$, a contradiction. We conclude that
$r(\beta) > 0$ as desired.
\end{proof}

\begin{lemma} \label{L:complete pseudometric}
The pseudometric $d$ on the set $\{u \in \tilde{\gotho}: \tilde{\alpha}(u) \leq p^{-1}\}$ 
is complete. That is, for every sequence $u_0,u_1,\dots$
with $\lim_{i,j \to \infty} d(u_i,u_j) = 0$, there exists $u$ for which
$\lim_{i \to \infty} d(u,u_i) = 0$.
\end{lemma}
\begin{proof}
Let $U_0$ denote the original sequence. For $h=0,1,\dots$, we produce an infinite subsequence
$U_{h+1}$ of $U_h$ such that any two elements $v_1,v_2$ of $U_{h+1}$
satisfy $\tilde{\alpha}(v_1 - v_2) \leq p^{-p^{h+1}/(p-1)}$.
To produce $U_1$, apply Lemma~\ref{L:compute metric}(a).
Given $U_h$ for some $h>0$, by Lemma~\ref{L:compute metric}(c),
$U_h$ falls into $p$ residue classes
modulo elements of $\tilde{\gotho}$ of norm less than $p^{-p^h/(p-1)}$.
In particular, one of these residue classes contains infinitely many terms of $U_h$;
by Lemma~\ref{L:compute metric}(b), all but finitely many such terms
are pairwise congruent modulo elements of $\tilde{\gotho}$ of norm at most $p^{-p^{h+1}/(p-1)}$.
We can thus choose these to constitute $U_{h+1}$.

By diagonalizing, we obtain a Cauchy sequence in $\tilde{\gotho}$ with respect to $\tilde{\alpha}$, which 
then has a limit $u$.
With respect to $d$, the original sequence is Cauchy and $u$ is a limit of a subsequence,
so it is also a limit of the entire sequence.
\end{proof}

In the manner of Corollary~\ref{C:residual}, we can describe
the residual extensions and norm groups
of points in $\mu^{-1}(\alpha)$. 
\begin{cor} \label{C:residual2}
Let $\beta$ be a point of $\mu^{-1}(\alpha)$,
classified according to Theorem~\ref{T:classify2}.
Let $|\alpha^\times|,|\beta^\times|$ denote the groups of nonzero values assumed
by $\alpha,\beta$, respectively.
Put $K = \Frac(\gotho)$.
\begin{enumerate}
\item[(i)]
For $\beta$ of type (i),
$\kappa_{\calH(\beta)}$ is algebraic over $\kappa_K$,
and $|\beta^\times|/|\alpha^\times|$ is a torsion group.
\item[(ii)]
For $\beta$ of type (ii),
$\kappa_{\calH(\beta)}$ is finitely generated over $\kappa_K$
of transcendence degree $1$,
and $|\beta^\times|/|\alpha^\times|$ is a finite group.
\item[(iii)]
For $\beta$ of type (iii),
$\kappa_{\calH(\beta)}$ is a finite extension of $\kappa_K$,
and $|\beta^\times|/|\alpha^\times|$ is a finitely generated abelian
group of rank $1$.
\item[(iv)]
For $\beta$ of type (iv),
$\kappa_{\calH(\beta)}$ is algebraic over $\kappa_K$,
and $|\beta^\times|/|\alpha^\times|$ is a torsion group.
\end{enumerate}
\end{cor}
\begin{proof}
By Ostrowski's theorem again (see \eqref{eq:ostrowski}),
in cases (ii) and (iii), it is enough to check the claims after
replacing $K$ by a finite
extension; in cases (i) and (iv), 
we may replace $K$ by a completed algebraic closure.
We make these assumptions hereafter.

In cases (i), (ii), (iii),
we have $\beta = \beta_{u,t}$ with $u \in \gotho$ and $\alpha(u) \leq p^{-1}$. 
For each $x \in W(\gotho)$,
by Lemma~\ref{L:stable presentation}, in $W(\tilde{\gotho})$
there exists a stable presentation
$x_0,x_1,\dots$ of $x$ with respect to $u$.
Let $\overline{x_i} \in \tilde{\gotho}$ denote the reduction modulo $p$
of $x_i$.
By Theorem~\ref{T:construct Gauss}(b),
\begin{equation} \label{eq:norm formula}
\beta_{u,t}(x) = \max_i \{(t/p)^i \tilde{\alpha}(\overline{x_i})\}.
\end{equation}
Consequently, in cases 
(i) and (ii), $|\beta^\times|/|\alpha^\times|$ is trivial;
in case (iii), $|\beta^\times|/|\alpha^\times|$ is freely generated by 
$t/p$.

In case (i), $\kappa_{\calH(\beta)}$ may be identified with the quotient
of $W(\gotho)$ by the ideal $(p, [u])$, so $\kappa_{\calH(\beta)} = \kappa_K$.
In case (ii), pick $v \in \gotho$ with $\alpha(v) = t/p$. For
$R = W(\gotho)[[v]^{-1}]$, we have $\gotho_R/\gothm_R \cong \kappa_K[z]$
for $z$ the class of $[v]^{-1}(p - [u])$. Consequently, 
$\kappa_{\calH(\beta)} = \kappa_K(z)$.
In case (iii), if $x \neq 0$,
then the maximum in \eqref{eq:norm formula} is only achieved by a single index
$i$. For this $i$, we have $\beta_{u,t}(x - (p - [u])^i [\overline{x_i}]) < \beta_{u,t}(x)$;
it follows that each element of $\kappa_{\calH(\beta)}$ is represented by
a Teichm\"uller element. Consequently, $\kappa_{\calH(\beta)} = \kappa_K$.

In case (iv), 
by Corollary~\ref{C:infimum2}, for each $y \in W(\gotho)$,
any sufficiently small $s \in (r(\beta),1]$ satisfies
$H(\beta,s)(y) = \beta(y)$.
If we choose $s \in |\alpha^\times|$, 
we deduce that $|\beta^\times|/|\alpha^\times|$ is trivial.
If we choose $s \notin |\alpha^\times|$, then
for any $z \in W(\gotho)$ with $\beta(z) \leq \beta(y)$, by case (iii),
there must exist $\lambda \in \gotho$ for which $H(\beta,s)(z
- [\lambda] y) < H(\beta,s)(y)$.
This implies 
\[
\beta(z -[\lambda] y) \leq H(\beta,s)(z - [\lambda] y)
< H(\beta,s)(y) = \beta(y),
\]
so $z/y$ and $[\lambda]$ have the same
image in $\kappa_{\calH(\beta)}$. Hence $\kappa_{\calH(\beta)} = \kappa_K$.
\end{proof}

\begin{remark}
One could also consider points of $\mu^{-1}(\alpha)$ obtained
by restricting points of $\mu^{-1}(\tilde{\alpha})$ of the form
 $H(\tilde{\alpha},\pi,t)$ for $\pi \in W(\tilde{\gotho})$ 
as in Hypothesis~\ref{H:primitive},
i.e., $\pi = \sum_{i=0}^\infty p^i [\overline{\pi_i}]$ with
$\tilde{\alpha}(\overline{\pi_0}) \leq p^{-1}$ and
$\tilde{\alpha}(\overline{\pi_1}) = 1$.
However, by Theorem~\ref{T:factor},
any such $\pi$ generates the same ideal as $p - [u]$ for some $u \in 
\tilde{\gotho}$, so $H(\tilde{\alpha},\pi,t) = H(\tilde{\alpha},u,t)$
for $t \in [0,1]$. Consequently, these points are again of types
(i), (ii), (iii) in Theorem~\ref{T:classify2}, and not type (iv).
\end{remark}

\end{document}